\documentclass[psamsfont]{amsart}

\usepackage{graphicx}
\usepackage{amssymb}
\usepackage{xypic}
\usepackage{hyperref}
\usepackage{amsmath,amssymb}
\usepackage{dsfont}\let\mathbb\mathds
\usepackage[english]{babel}

\newtheorem{theorem}{Theorem}[section]
\newtheorem{lemma}[theorem]{Lemma}
\newtheorem{corollary}[theorem]{Corollary}
\newtheorem{proposition}[theorem]{Proposition}

\theoremstyle{remark}
\newtheorem{rem}[theorem]{Remark}

\theoremstyle{definition}
\newtheorem{definition}[theorem]{Definition}

\DeclareMathOperator{\im}{Im}

\usepackage{marvosym}

\def\del{\partial}              

\def\bC{\mathbb C}          
\def\bR{\mathbb R}          
\def\bQ{\mathbb Q}          
\def\bN{\mathbb N}          
\def\bZ{\mathbb Z}          
\def\bT{\mathbb T}          
\def\bP{\mathbb P}

\def\bcp{\mathbb C \mathbb P}

\def\mF{\mathcal{F}}            
\def\mS{\mathcal{S}}            
\def\mL{\mathcal{L}}            
\def\mS{\mathcal{S}}            

\def\mO{\mathcal{O}}

\def\Span{\mbox{span}}

\def\bfE{\mbox{{\bf E}}}

\def\bfH{\mbox{{\bf H}}}

\def\bfN{\mbox{{\bf N}}}

\def\kt{\mathfrak{t}}

\def\ra{\rightarrow}

\def\vol{d\varpi}
\def\length{d\sigma}
\def\zeta{{\rm A}}
\def\del{\partial}
\def\Polytope{ P}

\def\mu{ x}

\newcommand{\defin}[1]{\textit{#1}}

 \newcommand*{\quot}[2]%
{\ensuremath{%
   \raisebox{.35ex}{\ensuremath{#1}}\big/\raisebox{-.35ex}{\ensuremath{#2}}}}

\newcommand{\bigO}{\large{O}}

\begin{document}

\title{Toric K\"ahler--Einstein metrics and convex compact polytopes}
\author{Eveline Legendre}
\date{\today}
\address{I.M.T., Universit\'e Paul Sabatier, 31062 Toulouse cedex 09, France }
\email{eveline.legendre@math.univ-toulouse.fr}



\thanks{The author is grateful to Professor Vestislav Apostolov for bringing her attention to this problem and for his numerous advice. This project ended during a stay at the MIT, the author thanks Professor Victor Guillemin for his interest in this problem. She also thanks Yanir Rubinstein for his explanations of some subtleties concerning singular K\"ahler metrics. Finally, she thanks the referees for pointing to her a mistake in a previous version
of this work and giving her ideas to complete a proof.}

\maketitle

\begin{abstract}
{ We show that any compact convex simple lattice polytope is the moment polytope of a K\"ahler--Einstein orbifold, unique up to orbifold covering and homothety. We extend the Wang--Zhu Theorem~\cite{WZ} giving the existence of a K\"ahler--Ricci soliton on any toric monotone manifold on any compact convex simple labelled polytope satisfying the combinatoric condition corresponding to monotonicity. We obtain that any compact convex simple polytope $P\subset \bR^n$ admits a set of inward normals, unique up to dilatation, such that there exists a symplectic potential satisfying the Guillemin boundary condition (with respect to these normals) and the K\"ahler--Einstein equation on $P\times \bR^n$. We interpret our result in terms of existence of singular K\"ahler--Einstein metrics on toric manifolds.}
\end{abstract}

\section{Introduction}

The question of existence of K\"ahler--Einstein metrics on compact complex manifold has been subject of intense investigations for the last decades. This problem makes sense on a compact complex manifold $(M^{2n},J)$ with a given K\"ahler class $\Omega\in H^{2}_{dR}(M)$ for which there is $\lambda\in \bR$ such that $\lambda\Omega=2\pi c_1(M)$. The case $\lambda\leq 0$ is non-obstructed and the existence of a K\"ahler--Einstein metric $(g,\omega)$, with $\omega\in\Omega$ was proved forty years ago~\cite{aubin,Y}. The case $\lambda>0$ proved to be a more difficult question, recently related to a certain notion of stability~\cite{CDS1,CDS2,CDS3,TianProof} and for which there are various known obstructions, notably the Futaki invariant~\cite{futaki}.
In the toric case (the K\"ahler structure is invariant by the Hamiltonian action of a real torus of dimension $n=\dim_{\bC}M$), it follows from the Wang--Zhu Theorem~\cite{WZ} which has been extended to orbifolds~\cite{SZ}, that the only obstruction to the existence of K\"ahler--Einstein metrics on monotone symplectic toric orbifolds (in the sense that there exists $\lambda>0$ such that $\lambda[\omega]=c_1(M)$) is the vanishing of the Futaki invariant. Through the toric correspondence, finding such orbifolds is a combinatorial problem on labelled polytopes. In the first part of this paper, we prove that any polytope can be labelled to satisfy these two conditions. To give a precise statement, we now recall the main lines of the correspondence.

Symplectic toric compact orbifolds are classified by {\it rational} labelled polytopes via the Delzant--Lerman--Tolman correspondence~\cite{delzant:corres,LT:orbiToric}. A labelled polytope is a pair $(P,\nu)$ where $P$ is a simple bounded convex polytope, open in a $n$--dimensional vector space $\kt^*$, $\nu=\{\nu_1,\dots,\nu_d\}\subset \kt$ is a set of vectors, inward to $P$, such that if we denote $F_1$,$\dots$, $F_d$ the facets (codimension $1$ face) of $P$, the vector $\nu_k$ is normal to $F_k$ for $k=1,\dots,d$ where $d$ is the number of facets.  The {\it defining functions} of a labelled polytope $(P,\nu)$ are the affine-linear functions $L_1,\dots, L_d$ on $\kt^*$ such that\footnote{ Our convention is that $P$ is open and we denote $\overline{P}=\{p\in\kt^*\,|\,L_k(p)\geq 0\}$. $\overline{P}$ is compact.} $P=\{p\in\kt^*\,|\,L_k(p)>0\}$ and $dL_k=\nu_k$. A rational labelled polytope $(P,\nu,\Lambda)$ is a labelled polytope $(P,\nu)$ and $\Lambda$ a lattice in $\kt$ such that $\nu\subset\Lambda$.

 \begin{rem} If $(P,\nu,\Lambda)$ is rational, there are (uniquely determined) positive integers $m_1,\dots,m_d$ such that $\frac{1}{m_i}\nu_{i}$ are primitive elements of $\Lambda$. Then $(P, m_1,\dots, m_d)$ is a \defin{rational labelled polytope} in the sense of Lerman--Tolman \cite{LT:orbiToric}.
  \end{rem}

For a given symplectic toric compact orbifold $(M,\omega,T)$, $\kt$ is the Lie algebra of the torus $T=\quot{\kt}{\Lambda}$ and the closure $\overline{P}$ is the image of the moment map. The symplectic properties are encoded in the data $(P,\nu)$. Notably, see~\cite{don:Large}, monotone symplectic toric orbifolds correspond to what we will call {\it monotone labelled polytopes}.
\begin{definition}\label{defMONOTONE}
 We say that $(P,\nu)$ is {\it monotone} if there exists $p\in P$ such that $L_1(p)=L_2(p)=\dots=L_d(p)$. In that case, we call $p$ the {\it preferred point} of $(P,\nu)$. \end{definition}

The space of invariant K\"ahler metrics on $M$ is parameterized by a subspace of convex functions on $P$, the set of symplectic potentials $\mS(P,\nu)$, see~\cite{abreuOrbifold,H2FII,don:extMcond}, whose definition we precisely recall in~\S\ref{subSectBC}. The scalar curvature of the metric $g_u$, associated to $u\in\mS(P,\nu)$, is given by the Abreu formula \begin{equation}\label{abreuForm} S(u)=-\sum_{i,j=1}^n \frac{\del^2 u^{ij}}{\del x_i\del x_j}\end{equation}
where $(x_1,\dots,x_n)$ are coordinates on $\kt^*$ and $u^{ij}=(\mbox{Hess }u)^{-1}$, see~\cite{abreu,abreuOrbifold}.

The {\it extremal affine function} of $(P,\nu)$, denoted $\zeta_{(P,\nu)}$, is an affine--linear function on $\kt^*$ which corresponds to the Futaki invariant~\cite{futaki} restricted to the (real) Lie algebra of the torus (the symplectic counterpart of the Futaki invariant as introduced in \cite{lejmi}). In particular, $\zeta_{(P,\nu)}$ is constant if and only if the Futaki invariant vanishes on $\kt$.  The extremal affine function is a useful invariant of $(P,\nu)$ since it satisfies
\begin{itemize}
  \item[-] $g_u$ is extremal, in the sense of Calabi~\cite{calabi}, if and only if $S(u)=\zeta_{(P,\nu)}$
  \item[-] $\zeta_{(P,\nu)}$ is constant, should a constant scalar curvature $T$--invariant compatible K\"ahler (cscK) metric exist.
\end{itemize}

In this paper, we prove the following statement.
\begin{theorem} \label{theo1} Given a compact simple convex polytope $\overline{P}$, there exists a set of normals $\nu$, unique up to dilatation, such that $(P,\nu)$ is monotone and has a constant extremal affine function. \end{theorem}

In dimension $2$, the existence of such labelling follows from elementary considerations~\cite{don:extMcond}.
\begin{rem} As noticed in~\cite{don:scalar}, labelled polytopes with constant extremal affine function are those for which the centers of mass of $(P,\vol)$ and $(\del P,\length_{\nu})$ coincide, where $\length_{\nu}$ is the volume form on $\del P$ such that $\nu_k\wedge\length_{\nu}=-\vol$ on the facet $F_k$. The set of normals $\nu$ given by Theorem~\ref{theo1} is characterized by the fact that the preferred point of $(P,\nu)$ (as a monotone labelled polytope) coincides with the center of mass of $(\del P,\length_{\nu})$. This last characterization was proved by Mabuchi~\cite{mabuchi} and used to classify toric complex surfaces admitting a compatible K\"ahler--Einstein metric : $\bcp^2, \bcp^1\times \bcp^1$ and $\bcp^2 \#3 \bcp^2$.
\end{rem}

We will see that the set of normals $\nu$ given by Theorem~\ref{theo1} can be included in a lattice if and only if $P$ is a lattice polytope (i.e whose vertices lie in a lattice) and thus, using the Theorem of Wang--Zhu/Shi--Zhu~\cite{WZ,SZ} we get

\begin{corollary} \label{question1} Every (simple convex compact) lattice polytope is the moment polytope of a compact K\"ahler--Einstein toric orbifold, unique up to dilatation or orbifold covering. \end{corollary}

The case where the set of normals $\nu$ given by Theorem~\ref{theo1} is not rational motivates us to extend the Wang--Zhu Theorem for general labelled polytopes. More precisely, Wang and Zhu showed in~\cite{WZ} that any Fano toric manifold $(M^{2n},J,T)$ admits a K\"ahler--Ricci soliton $(g,Z)$, that is, a K\"ahler metric $g$ and a holomorphic vector field $Z$ such that
\begin{equation}\label{KRSequation}
 \rho^g -\lambda\omega =\mL_Z\omega
\end{equation}
where $\rho^g$ is the Ricci form of $g$, $\omega$ the K\"ahler form $\omega= g(J\cdot,\cdot)$ and $\lambda = \frac{1}{2n}\,\overline{\mathrm{Scal}}$ with $\overline{\mathrm{Scal}}=\int_{M} \mathrm{Scal}\, \omega^{n}/\int_M\omega^n$. In that case, $2\pi c_1(M)=\lambda [\omega]$. The vector field $Z$ is uniquely determined by the data $(M,[\omega], T)$ as follows: denoting $p$ the preferred point of the monotone labelled polytope $(P,\nu)$ associated to $(M,[\omega],T)$, there is a unique linear function on $\kt^*$, $a\in\kt$, such that
\begin{equation}\label{RSvector}
  \int_P e^{2a} (f-f(p))\vol =0
\end{equation} for all $f\in\mbox{Aff}(P,\bR)$. If a holomorphic vector field $Z$ satisfies~\eqref{KRSequation}, then $Z= JX_a-iX_a$, see~\cite{TZ,WZ}. The case $a=0$ implies that $Z=0$ and the K\"ahler-Ricci soliton of Wang--Zhu is a K\"ahler--Einstein metric.\\

According to the work of Donaldson~\cite{don:Large}, a symplectic potential $u\in\mS(\Polytope,\nu)$ corresponds to a K\"ahler--Ricci soliton with respect to $\lambda>0$ and $a\in\kt$ if and only if
\begin{equation}\label{eq:KRsolitonPOT}
  \frac{1}{2}\log \det (\mbox{Hess }u)_x +\lambda h(x)= a(x)
\end{equation}
where $h$ is the Legendre transform of $u$ (seen as a function on $P$, via the change of variable $x\mapsto (d u)_{x}\in\kt$) and the preferred point of $(P,\nu)$ is the origin.
%

In fact, the argument of Wang--Zhu holds for any labelled polytope without any deep modification. In order to find appropriate scope for extending their proof in the case when $(P,\nu)$ is non necessarily rational, we consider $P\times \kt$ with its symplectic structure (that is, $P\times \kt \subset \kt^*\times \kt\simeq T^*\kt$) and the $\kt$--Hamiltonian action by translation on the second factor, the moment map being the projection on the first factor. The invariant K\"ahler metric $g_u$, for a symplectic potential $u\in\mS(P,\nu)$, is simply a $\kt$--invariant K\"ahler metric on $P\times \kt$ with specific behavior along $\del P\times \kt$. As introduced in \cite{DuistPelayo}, see also~\cite{don:Large,don:extMcond} and \S\ref{subSectCPT}, for each vertex $p$ of $P$ there is an open toric symplectic manifold $(M_p,\omega_p,T_p)$ depending only on $(P,\nu)$. In the rational case, $(M_p,\omega_p,T_p)$ is a uniformizing chart for the orbifold. The boundary condition on symplectic potentials corresponds to the fact that $g_u$ defines a smooth metric on each of the manifolds $(M_p,\omega_p,T_p)$, see~\S\ref{subSectBC}. In Section~\ref{secKRSorbi}, we notice that the test functions appearing in the proof of Wang--Zhu, behave as functions defined on the compact set $\overline{P}$ while the boundary condition, suitably interpreted, allows us to apply the (local) computations of Yau~\cite{Y} and Tian--Zhu~\cite{TZ} on each chart $(M_p,\omega_p,T_p)$. Along the way, we have to show that both Yau's Theorem~\cite{Y} and Zhu's Theorem~\cite{Z} hold, suitably interpreted, in this extended setting.  

\begin{theorem} \label{WZtheoGLP} Let $(P,\nu)$ be a monotone labelled polytope with preferred center $0\in\kt^*$ and compact closure $\overline{P}$. There exists a solution $u\in \mS(P,\nu)$ of equation~\eqref{eq:KRsolitonPOT}, so that $g_u$ is a $\kt$--invariant K\"ahler--Ricci soliton on $P\times \kt$. This solution $u$ is unique in $\mS(P,\nu)$ up to addition of an affine-linear function and $g_u$ is K\"ahler--Einstein if and only if $\zeta_{(P,\nu)}$ is constant.\end{theorem}

A result of Donaldson~\cite{don:extMcond} implies that the set of normals for which there exists a solution of the Abreu equation is open in the set of inward normals of a fixed polytope, see~\S\ref{openNORMALS}. Together with Theorems~\ref{theo1} and~\ref{WZtheoGLP}, it yields to

\begin{corollary}\label{propEXISTENCEextremal}
 For each $n$--dimensional polytope $P$, there exists a non-empty open set $\bfE(P)$ of inward normals $\nu$ for which there exists an extremal toric K\"ahler metric $g_u$ with $u\in\mS(P,\nu)$. Moreover, $\bfE(P)$ contains a codimension $n$ subset corresponding to cscK metrics and contains the $1$--dimensional cone of toric K\"ahler--Einstein metrics. In particular, if there exists a lattice for which $P$ is rational then there exist extremal toric K\"ahler orbifolds with moment polytope $\overline{P}$.
\end{corollary}

A compact toric symplectic orbifold associated to a rational labelled polytope $(P,\nu,\Lambda)$ is a compactification of $P\times T$ where $T=\kt/\Lambda$ and $P\subset \kt^*$. Consequently, as a straightforward application of Theorems~\ref{theo1} and~\ref{WZtheoGLP}, we obtain in~\S\ref{sesctSINGULAR}, for any smooth compact toric symplectic manifold, the existence of a toric K\"ahler--Einstein metric $g_{KE}$ on the open dense subset where the torus acts freely. The singular behavior of $g_{KE}$ along the pre-image of the interior of a facet $F_k$ is conical of angle $2\pi a_k$ where $a_k$ is the ratio between the normal $\nu_k$ and the normal to $F_k$ found in Theorems~\ref{theo1}. Here, the complex structure and the metric are singular while the symplectic structure is smooth. Using the standard procedure (with Legendre transform) to complexifies $P\times T$ the singularity lies along the pre-image of the boundary and we can pass to the more usual setting in the study of conical singularity of metrics, 
e.g.~\cite{Do:cone,Tsingular}, where the symplectic form and the metric are singular but not the complex structure.
%

Theorem~\ref{WZtheoGLP} provides an alternative proof of the Futaki--Ono--Wang Theorem~\cite{FutakiOnoWang}. This theorem establishes the existence of toric Sasaki--Ricci soliton on contact toric manifolds with a (fixed) Reeb vector field satisfying the two conditions :\\
  -- the basic first Chern form of the normal bundle of the Reeb foliation is positive,\\
  -- the first Chern class of the contact bundle is trivial. \\
Compact contact toric manifolds with a fixed Reeb vector field are in one-to-one correspondence with labelled polytopes whose defining functions lie in a lattice and satisfy a certain weaker condition than the Delzant condition, see~\cite{BG:book,moi:ENU,L:contactToric}. In this correspondence as well, a compatible toric Sasaki metric is given by a symplectic potential and the scalar curvature is given by the Abreu formula~\eqref{abreuForm}, up to an additive constant depending only on the dimension. The hypothesis of the Futaki--Ono--Wang Theorem corresponds to the fact that the associated labelled polytope is monotone, see~\cite{FutakiOnoWang,reebMETRIC}.

\section{Labelled polytopes and toric geometry}\label{sectBACKGROUND}
  For the purpose of this paper, we need to slightly reinterpret the geometry associated to a labelled polytope based on the approach~\cite{don:Large,don:extMcond,DuistPelayo}. In what follows, polytopes always refers to {\it simple bounded polytope} where {\it simple} means that each vertex is the intersection of no more than $n$ facets.  

\subsection{Symplectic toric orbifolds as compactifications}\label{subSectCPT}
Let $(M,\omega,T, \varrho)$ be a compact symplectic toric orbifold, that is, $\varrho : T\hookrightarrow \mbox{Ham}(M,\omega)$ embeds $T$ as a subgroup and $2\dim T=\dim M$. We denote $\kt=\mbox{Lie }T$. There is a moment map $$x:M\ra \kt^*$$ which is $T$--invariant and uniquely determined, up to addition of a constant, by the relation $d\langle x,a\rangle=- \omega(X_a,\cdot)$ where $X_a= d_{\mbox{\tiny{id}}}\varrho(a)$ is the vector field induced by the infinitesimal action of $a\in\kt$. The image of the moment map $\overline{P}=\im x$ is a compact convex simple polytope\footnote{We denote $P$ the interior of the polytope and $\overline{P}$ its closure. In this text, polytopes are always assumed to be convex and simple with compact closure.}. The weights of the action of the torus on the tangent spaces of fixed points determine a set of vectors $\nu=\{\nu_1,\dots,\nu_d\}\subset \kt$ normal to the facets of $P$ and lying in $\Lambda$, the lattice of circle subgroups of $T$, and thus makes $(P,\nu)$ a 
rational labelled polytope with respect to $\Lambda$ as defined in the introduction. The Delzant--Lerman--Tolman correspondence~\cite{delzant:corres,LT:orbiToric} states that the data $(P,\nu,\Lambda)$ characterizes $(M,\omega,T)$ up to a $T$--equivariant symplectomorphism.

In~\cite{DuistPelayo}, Duistermaat and Pelayo gave a way (alternative to the so-called Delzant construction~\cite{delzant:corres}) to build $(M,\omega,T)$ from the data $(\Polytope, \nu, \Lambda)$ in the smooth case, see also~\cite{don:Large}. The idea is based on the fact that $M$ can be seen as a compactification of $\Polytope\times T$ prescribed by the combinatorial data of $\Polytope$. We slightly adapt this construction here to cover the case of orbifolds and to see where it fails in the non rational case.

Given a labelled polytope $(\Polytope, \nu)$, we denote the set of (closed) faces of $\Polytope$ by $\mF(\Polytope)$. The facets of $P$ are still denoted $F_1,\dots, F_d\in\mF(\Polytope)$. For $F\in\mF(P)$, denote $I_{F}\subset\{1,\dots,d\}$, the set of indices such that $F=\bigcap_{k\in I_F} F_k$. For example, $\overline{\Polytope}\in \mF(\Polytope)$ and $I_{\overline{P}}=\emptyset$. For a vertex $p$, $I_{\{p\}}$ has $n$ elements and $\Lambda_p=\Span_{\bZ}\{\nu_k\,|\,k\in I_{\{p\}}\}$ is a lattice in $\kt$. For a face $F\in \mF_p(\Polytope)$, $T_F=\quot{\Span_\bR\{\nu_k\,|\,k\in I_F\}}{\Lambda_p\cap\Span_\bR\{\nu_k\,|\,k\in I_F\}}$ is a subtorus of $T_p=\quot{\kt}{\Lambda_p}$ if $p\in F$.

Given a vertex $p$ of $P$, we call $\mF_p(\Polytope)$ the set of faces containing $p$. For $F\in \mF_p(P)$, we denote $s_p(F)$ the subset of $F$ obtained by removing all the subfaces which does not contain $p$, that is $s_p(F) = \{x\,|\, x\in \mathring{E}, p\in E, E\subset F\}$ where $\mathring{E}$ is the interior of the face $E$ (in $E$). In particular, the interior of a vertex is the vertex itself. Thus, $\bigcup_{F\in \mF_p(P) } s_p(F)=\bigcup_{F\in\mF_p(P)}\mathring{F}$ is an open neighborhood of $p$ in $\overline{P}$. Set 
$$M_p= \quot{\bigsqcup_{F\in\mF_p(\Polytope)} (s_p(F)\times T_p/T_F)}{\sim}$$ where, for $(x,\theta) \in F\times T_p/T_F$ and $(x',\theta') \in F'\times T_p/T_{F'}$, $(x,\theta)\sim (x',\theta')$ if \begin{itemize}
  \item[1)] $x=x'$, and
  \item[2)] the equivalence classes of $\theta$ and $\theta'$ in $\quot{T_p}{T_{F\cap F'}}$ coincide.
\end{itemize}
Here, the first condition implies that $F\cap F'\neq \emptyset$, so $F\cap F'\in \mF_p(\Polytope)$ and $T_{F\cap F'}$ contains $T_F$ and $T_{F'}$ as subgroups. The second condition refers to the fact that $\quot{T_p}{T_{F\cap F'}}$ is the quotient of $T_p/T_F$ by $\quot{T_{F\cap F'}}{T_{F}}$ and the quotient of $T_p/T_{F'}$ by $\quot{T_{F\cap F'}}{T_{F'}}$.

Ordering the normals $\nu_{k_1},\dots,\nu_{k_n}$ ($k_i\in I_{\{p\}}$), we get an identification $T_p\simeq \bT^n=\quot{\bR^n}{\bZ^n}$ via which $T_p$ acts on $\bC^{n}$. For an equivariant neighborhood $U_p$ of $0\in \bC^n$, the map $\phi_p:U_p\ra M_p$, defined by
\begin{equation}
  \label{localEQUIV}\phi_p(z)= \left[\left(p + \frac{1}{2}|z_i|^2\nu_{k_i}^*, \left(e^{2\pi \sqrt{-1} \theta_1}, \dots ,e^{2\pi \sqrt{-1} \theta_n}\right)\right)\right]
\end{equation}
where $z=(|z_1|e^{2\pi \sqrt{-1} \theta_1},\dots, |z_n|e^{2\pi \sqrt{-1} \theta_n})$, is a well-defined (i.e does not depend on the choice of $e^{2\pi \sqrt{-1} \theta_i}$ when $|z_i|=0$) equivariant homeomorphism. The chart $\left(U_p, \phi_p\right)$ provides a (smooth) differential structure to $M_p$.

Now, the cotangent space of $T_p$ is naturally equipped with an exact symplectic form, the differential of the Liouville $1$--form, for which the action of $T_p$ on itself pull-backs to a Hamiltonian action. Given an equivariant trivialization $T^*T_p\simeq \kt^*\times T_p$, the product $P\times T_p$ inherits of the structure of Hamiltonian $T_p$--space whose moment map is simply the projection on the first factor. The chart above extends this structure to give a (non-compact) symplectic toric manifold $(M_p,\omega_p,T_p)$ with moment map $x:M_p \ra \kt^*$ so that $\mbox{Im } x= \bigcup_{F\in\mF_p(P)} \mathring{F}$, see \cite{don:Large,LT:orbiToric}.\\

When $(P,\nu,\Lambda)$ is rational, $\Lambda_p\subset \Lambda$ for all vertex $p$ and the quotient of $T_p$ by the finite subgroup $\quot{\Lambda}{\Lambda_p}$ is the torus $T=\quot{\kt}{\Lambda}$. The quotient map $q_p:T_p\ra T$ gives a way to glue $M_p$ to $P\times T$ providing an orbifold uniformizing chart with structure group $\quot{\Lambda}{\Lambda_p}$. Doing that on all vertices, we obtain the compact symplectic toric orbifold, $(M,\omega,T)$, associated to $(P,\nu,\Lambda)$ with moment map $x:M\ra \kt^*$.
\begin{definition}\label{DelzantCOND}\cite{delzant:corres}
  A rational labelled polytope $(P,\nu,\Lambda)$ is {\it Delzant} if, for each vertex $p$, $\Lambda_p=\Lambda$. In particular, $(P,\nu,\Lambda)$ is Delzant if and only if the associated symplectic toric orbifold is a manifold (all orbifold structure groups are trivial).
\end{definition}
\begin{rem} \label{LatticeChoice}Taking a bigger lattice $\Lambda\subset \Lambda'$, corresponds to taking the global quotient by the finite group $\quot{\Lambda'}{\Lambda}$, see~\cite{H2FII}.\end{rem}

\begin{rem} If $(P,\nu,\Lambda)$ is rational, we can replace $\Lambda_p$ by $\Lambda$ in the definition of $T_F$ and set  $$|M|= \quot{\bigsqcup_{F\in\mF(\Polytope)} (F\times T/T_F)}{\sim}$$ with the same equivalence relation as above.
The topological space $|M|$ is the underlying topological space of $M$ and is a compactification of $\Polytope\times T$. The choice of a labelling specifies an orbifold structure on $|M|$ but $|M|$ does not depend on it.
\end{rem}

\subsection{Action-angle coordinates}
To a convex polytope $P\subset \kt^*$ one can associate a symplectic manifold $(P\times \kt,dx\wedge d\theta)$ where $x=(x_1,\dots, x_n) : \kt^*\ra \bR^n$ and $\theta=(\theta_1,\dots, \theta_n): \kt\ra \bR^n$ are any sets of affine coordinates and $dx\wedge d\theta = \sum_{i=1}^n dx_i\wedge d\theta_i$ is (trivially) a symplectic form. More intrinsically, one could consider $T^*P$, the cotangent of the polytope itself (recall that $P$ is open in $\kt^*$), endowed with its canonical symplectic structure. 
The action of $\kt$ on $P\times \kt$ by translation on the second factor is Hamiltonian with moment map $x:P\times \kt\ra \kt^*$.

\begin{rem} One can choose $(x,\theta)$ to be {\it dual coordinates} (or {\it canonical coordinates}) on $\kt^*\times\kt^*$, defined for a given basis $e_1,\dots, e_n$ of $\kt$ as $x_i=\langle x, e_i\rangle$ and $\theta_i = \langle \theta, e^*_i\rangle$ where $e_1^*,\dots, e^*_n$ is the dual basis. Doing so would imply that $dx\wedge d\theta$ is a {\it canonical} symplectic form on $\kt\times \kt^*$ corresponding to the differential of the Liouville form on $T^*\kt=\kt\times \kt^*=T^*\kt^*$. However, this is not essential for our purpose since the obvious change of coordinates (say from $\theta$ to $\theta'$) on the second factor identify the structures ($dx\wedge d\theta$ and $dx\wedge d\theta'$).     
\end{rem} 

Given a compact symplectic toric orbifold $(M,\omega,T, \varrho)$ associated to $(P,\nu,\Lambda)$, the {\it action-angle coordinates} are local coordinates on $\mathring{M}=x^{-1}(P)$ (the subset of $M$ where $T$ acts freely) identifying locally $(\mathring{M},\omega_|)$ with $(P\times \kt,dx\wedge d\theta)$. Usually, the existence of such coordinates is proved using a compatible toric K\"ahler structure which is known to exist by the Delzant construction, see~\cite{abreuOrbifold,H2FII,CDG}. In view of the construction presented in \S~\ref{subSectCPT}, it is obvious that there is a $T$--equivariant symplectomorphism between $(\mathring{M},\omega_{|_{\mathring{M}}})$ and $(P\times T, dx\wedge d\theta)$. The universal cover of $\mathring{M}$, endowed with the symplectic form induced from $\omega_{|_{\mathring{M}}}$, is symplectomorphic to $(P\times \kt,dx\wedge d\theta)$. The action--angle coordinates are $(x,\theta)$ but seen as local coordinates on $\mathring{M}$ on which they satisfies $d\theta_i(X_j):=d\
\theta_i(d_{\mbox{\tiny{id}}}\varrho(e_j))=\delta_{ij}$.
\begin{rem} Two compact symplectic toric orbifolds $(M',\omega',T',\varrho')$, $(M,\omega,T,\varrho)$ associated to the same polytope $P$ (assuming the Lie algebra $\kt$ and $\kt'$ are identified but not the lattices) share the same action-angle coordinates in the sense that the symplectic manifolds $(\mathring{M}',\omega'_{|_{\mathring{M}'}})$ and $(\mathring{M},\omega_{|_{\mathring{M}}})$ have a common universal cover on wich $\omega'$ and $\omega$ pull back as the same symplectic structure (up to a diffeomorphism). \end{rem}
     
%
%

 \begin{proposition}\label{LOCALmetric} \cite{abreu} For any strictly convex function $u\in C^{\infty}(P)$, the metric
\begin{align}\label{ActionAnglemetric}
g_{u} = \sum_{i,j} G_{ij}dx_i\otimes dx_j + H_{ij}d\theta_i\otimes d\theta_j,
\end{align}
with $(G_{ij})=\mbox{Hess }u$ and $(H_{ij})=(G_{ij})^{-1}$, is a smooth K\"ahler structure on $P\times \kt$ compatible with the symplectic form $dx\wedge d\theta$. Conversely, any $\kt$--invariant compatible K\"ahler structure on $(P\times \kt, dx\wedge d\theta)$ is of this form.
 \end{proposition}

 \subsection{The boundary condition}\label{subSectBC}
Here again $(P,\nu)$ is a labelled polytope (with $\overline{P}$ compact, convex and simple) and the functions $L_1,\dots, L_k$ are the affine-linear functions defining $(P,\nu)$ as $dL_k=\nu_k$ and $P=\{x\in\kt^*| L_k(x)>0,\, k=1,\dots,d\}$.

\begin{definition}
  A symplectic potential of $(P,\nu)$ is a continuous function $u\in C^0(\overline{P})$ whose restriction to $P$ or to any face's interior (except vertices), is smooth and strictly convex, and $u-u_o$ is the restriction of a smooth function defined on an open set containing $\overline{P}$ where $$u_o = \frac{1}{2}\sum_{k=1}^d L_k \log L_k$$ is the {\it Guillemin potential}. We denote by $\mS(P,\nu)$, the set of symplectic potentials.
\end{definition}
The Guillemin potential is a symplectic potential corresponding to the {\it Guillemin metric}~\cite{guillMET}. Denote $\mbox{Aff}(P,\bR)$ the space of real valued affine-linear functions on $P$.
\begin{proposition}
  \cite{H2FII, don:extMcond}
   The set of smooth compatible toric (orbifold) K\"ahler metrics on $(M,\omega,T)$ is in one-to-one correspondence with the quotient of $\mS(P,\nu)$ by ${\rm Aff}(P,\bR)$, acting by addition. The correspondence is explicit and given by \eqref{ActionAnglemetric}.
\end{proposition}

The smooth compactification of a metric is a local issue. Even though $(P,\nu)$ might not be rational (for any lattice), a symplectic potential $u\in \mS(P,\nu)$ defines, via  \eqref{ActionAnglemetric}, a K\"ahler metric $g_u$ on $P\times \kt$ which is $\kt$--invariant and thus, for any vertex $p$, defines a K\"ahler metric, still denoted $g_u$, on $P\times T_p$. The boundary condition implies that $g_u$ is the restriction to $P\times T_p$ of a smooth $T_p$--invariant K\"ahler metric on $(M_p,\omega_p,T_p)$. Recall that $(M_p,\quot{\Lambda}{\Lambda_p})$ is an orbifold uniformizing chart near the pre-image of a vertex in the rational case and that smooth orbifold metrics are defined as metrics which may be lifted as smooth metrics on a chart.\\

Apostolov--Calderbank--Gauduchon--T{\o}nnesen-Friedman gave the following alternative description of the boundary condition.
\begin{proposition}\label{propHcondbord}\cite{H2FII} Given a labelled polytope $(P,\nu)$, a strictly convex function $u\in C^{\infty}(P)$ is a symplectic potential of $(P,\nu)$ if and only if, denoting $\bfH =(\mbox{Hess } u)^{-1}$,
\begin{itemize}
  \item $\bfH$ is the restriction to $P$ of a smooth $S^2\mathfrak{t}^*$--valued function on $\overline{P}$,
  \item for every $k=1,\dots, d$, for every $y$ in the interior of the facet $F_k$,
\begin{equation}\label{condCOMPACTIFsasak}
  \bfH_y (\nu_k, \cdot) =0\;\;\;\mbox{ and }\;\;\; d\bfH_y (\nu_k, \nu_k) =2\nu_k,
\end{equation}
\item the restriction of $\bfH$ to the interior of any face $F \subset P$ is a positive definite $S^2(\mathfrak{t}/\mathfrak{t}_F)^*$--valued function.\end{itemize} \end{proposition}
%
%

 \subsection{The curvature and the extremal affine function}

Fixing any euclidian volume form $\vol$ on $\kt^*$, Donaldson~\cite{don:scalar} pointed out that the $L^2(P,\vol)$--projection of the scalar curvature $S(u)$, given by \eqref{abreuForm}, on the space of affine linear functions, $\mbox{Aff}(P,\bR)$, is independent of the choice of $u\in\mS(P,\nu)$. The resulting projection $\zeta_{(P,\nu)}\in \mbox{Aff}(P,\bR)$ is the {\it extremal affine function} of $(P,\nu)$. Indeed, integrating~(\ref{abreuForm}) by part, the boundary condition of symplectic potentials gives $$\frac{1}{2}\int_{\Polytope} S(u) x_i\vol =\!\int_{\del \Polytope} x_i\length_{\nu}=: Z_i(\Polytope,\nu)$$ where $\length_{\nu}$, when restricted to any facet $F_k$, is a $(n-1)$--form defined by $\nu_k\wedge \length_{\nu}=-\vol$. 

Choose a basis $(e_1,\dots, e_n)$ of $\mathfrak{t}$ and set $x_0= 1$, $x_1=\langle e_1,\cdot\rangle$, $\dots$, $x_n=\langle e_n,\cdot\rangle $ of $\mbox{Aff}(\Polytope,\bR)$. The extremal affine function of $(\Polytope,\nu)$ is $\zeta_{(\Polytope,\nu)}=\sum_{i=0}^n \zeta_i x_i$ where the vector $\zeta = (\zeta_0,\dots,\zeta_n)\in \bR^{n+1}$ is the unique solution of the linear system:
\begin{equation}\label{systCHPext}
\begin{split}
\sum_{j=0}^n W_{ij}(\Polytope)\, \zeta_j &= 2Z_i(\Polytope,\nu),\;\;\; i= 0,\dots, n\\
\mbox{with }\quad W_{ij}(\Polytope) = \int_{\Polytope} x_ix_j\vol &\quad \mbox{ and }  \quad Z_i(\Polytope,\nu) =\!\int_{\del \Polytope} x_i \length_{\nu}.
 \end{split}
\end{equation}

 \subsection{Complex coordinates}\label{subsectCPLX}
For a K\"ahler structure $(g_u,dx\wedge d\theta,J_u)$ given by \eqref{ActionAnglemetric} on $P\times \kt$, the set $\{J_uX_1,\dots,J_uX_n, X_1,\dots,X_n,\}$ is a frame of real holomorphic commutative vector fields, which gives an identification $\kt\oplus\sqrt{-1}\kt \simeq T_{(x,\theta)}(P\times\kt)$ and (a priori) local holomorphic coordinates $z=t+\sqrt{-1}\theta$ where $dt_i = -d^c\theta_i$.

In the rational case, the complex coordinates $z=t+\sqrt{-1}\theta$ are only local on $\mathring{M}$ and are given by the exponential map, see~\cite{don:Large}. Actually, in this context, for a point $y\in\mathring{M}$, the tangent space $T_y\mathring{M}\simeq\kt\oplus\sqrt{-1}\kt \simeq \bC^n$ is naturally identified with the universal cover of $\mathring{M}$ where the covering map is just the exponential $$\mathring{M}\simeq\quot{\kt\oplus\sqrt{-1}\kt}{2\pi \sqrt{-1} \Lambda}\simeq (\bC^*)^n.$$

As explained in the literature see for e.g.~\cite{CDG,don:Large}, by writing $dx\wedge d\theta$ in the coordinates $z$, we find a K\"ahler potential \begin{equation}\label{Kpot}dx\wedge d\theta= \sum_{i,j=1}^n \frac{\del^2\phi}{\del t_i\del t_j}dt_i\wedge d\theta_j =dd^c\phi. \end{equation} In the rational case, $\phi$ is a globally defined function on $\kt$ via the identification provided by the exponential near a point of $\mathring{M}$. Changing the base point corresponds to translate $\phi$ by an affine-linear function of $\kt$. The correspondence between the symplectic potential $u$ and the K\"ahler potential $\phi$ is done via the Legendre transform: \begin{equation}\label{legTRANS}
   u(x) = \langle x,t\rangle-\phi(t)
 \end{equation} where $t$ is the unique point of $\kt$ such that $d\phi_t=x$ or inversely $x$ is the unique point of $\kt^*$ such that $du_x=t$. The image of the differential of the K\"ahler potential is the (open) polytope $P$ (i.e $P=\im (t\mapsto d\phi_t)$). 

A symplectic potential $u\in\mS(P,\nu)$ provides an identification
$$\Phi_u : P\times\kt \ra \kt\oplus\sqrt{-1}\kt$$
via its differential $du:P\ra \kt$ (which is a diffeomorphism since $u$ is strictly convex on $\mathring{M}$) so the coordinates $z=t+\sqrt{-1}\theta$ are globally defined on $P\times\kt$. In the rational case, this identification fits with the fact that both spaces are identified with the universal cover of $\mathring{M}$.
\begin{rem}The boundary condition on symplectic potentials is equivalent to the common asymptotic behavior of K\"ahler potentials. Using the identification $du:P\stackrel{\sim}{\ra}\kt$ or the inverse $d\phi:\kt\stackrel{\sim}{\ra}P$ one can express the boundary condition on $u$ as asymptotic behavior of the K\"ahler potential $\phi$ (recall that $\frac{\del^2\phi}{\del t_i\del t_j}(t) =(H_{ij}(x))$ whenever $du_x=t$). The Guillemin potential $u_o$ gives $du_o = \frac{1}{2} \sum_{k=1}^d (\log L_k +1)\nu_k.$ In particular, the normals determine the rate of divergence of $du_o$ when $x\ra \partial P$.
\end{rem}

 Distinct symplectic potentials $u, u_o\in\mS(P,\nu)$ lead to distinct K\"ahler structures on $\kt\oplus\sqrt{-1}\kt$ $$\left((\Phi_{u_o}^{-1})^*g_{u_o},\, \omega_o = (\Phi_{u_o}^{-1})^*dx\wedge d\theta\right)\;\;\;\mbox{ and }\;\;\;\left((\Phi_u^{-1})^*g_u,\, \omega = (\Phi_u^{-1})^*dx\wedge d\theta\right)$$ compatible with the same complex structure. Denoting $\phi$ and $\phi_o$ the Legendre transform of $u$ and $u_o$ respectively, we have $\omega-\omega_o = dd^c(\phi-\phi_o)$. Going back on $P\times\kt$, using $\Phi_{u_o}^{-1}$ we get \begin{equation}\label{CPLXpointofview}
  \left(g_{u_o},\, dx\wedge d\theta, J_{u_o}\right)\;\;\;\mbox{ and }\;\;\;\left((\Phi_u^{-1}\circ\Phi_{u_o})^*g_u,\, (\Phi_u^{-1}\circ\Phi_{u_o})^*dx\wedge d\theta, J
_{u_o}\right).
\end{equation} The map $\Phi_u^{-1}\circ\Phi_{u_o}$ is a $\kt$--invariant smooth diffeomorphism of $\overline{P}\times \kt$ fixing the boundary, thanks to the boundary condition on $u$ and $u_o$. In particular, the function $x\mapsto(\phi-\phi_o)(d(u_o)_x)$ is the restriction to $P$ of a smooth function on $\overline{P}$.

\section{Geometry in the non rational case}\label{sectGEOMnonRAT}

\subsection{Norms and integration}\label{subSECTcomment}

\begin{rem}
In what follows, any function $f$, defined on $P$ or on $\overline{P}$, is identified with its pull-back on $\overline{P}\times \kt$ and is also denoted $f$. On suitable subsets, we even identify $f$ with the corresponding $T_p$--invariant function on the chart $(M_p,\omega_p,T_p)$ of a vertex $p$.
\end{rem}

Fixing an orientation on $\kt^*$, a lattice $\Lambda\subset \kt$ naturally provides a volume form, say $\vol_{\Lambda}$, on $\kt^*$ since $\mbox{Gl}(\Lambda)\subset \mbox{Sl}(\kt)$. With the dual volume form $\vol_{\Lambda}^*$ on $\kt$, given by the dual lattice, the volume of the torus $T=\kt/\Lambda$ is $(2\pi)^n$. As noticed in~\cite{guillMET}, given a rational labelled polytope $(P,\nu,\Lambda)$ associated to the compact symplectic orbifold $(M,\omega,T)$ with moment map $x:M\ra \kt^*$, Fubini's Theorem implies that the integration of a $T$--invariant function on $V\subset M$ is the integration of the corresponding function on $x(V)\subset \overline{P}$ times the constant $(2\pi)^n$. Precisely, given an integrable function $f$ on $U\subset P$, assuming $x^{-1}(U)$ is covered by the orbifold uniformizing chart $\left( M_p,\quot{\Lambda}{\Lambda_p},\psi_p\right)$, see \S\ref{subSectCPT}, we have
\begin{equation}\label{intP}\begin{split}
\int_{U}f\vol_{\Lambda}  &= \frac{1}{(2\pi)^n} \int_{x^{-1}(U)} f \frac{\omega^n}{n!} = \frac{1}{|\Lambda/\Lambda_p|} \frac{1}{(2\pi)^n} \int_{\psi_p^{-1}(x^{-1}(U))} f \frac{\omega_p^n}{n!}\\
&= \frac{1}{\int_{T_p}\vol_{\Lambda}^*} \int_{\psi_p^{-1}(x^{-1}(U))} f \frac{\omega_p^n}{n!}.
\end{split}                                                                                                                                                                                                                                                                                                                                                                          
\end{equation}

If we consider only a labelled polytope $(P,\nu)$, there is no preferred lattice but we can arbitrarily choose a volume form $\vol=dx_1\wedge\dots \wedge dx_n$. Formula~\eqref{intP} still holds: for an integrable function $f$ defined on a neighborhood $U$ of a vertex $p$ of $P$, denoting $x:M_p\ra \kt^*$ the moment map of the chart $(M_p,\omega_p,T_p)$,
\begin{equation}\label{intP2}\int_Uf\vol = \frac{1}{c_p}\int_{x^{-1}(U)}f\frac{\omega_p^n}{n!}\end{equation}
where $c_p=\int_{T_p}\vol^* \in \bR_{>0}$ is a constant depending on $p$, on $\nu$ and on $\vol$. The value~\eqref{intP2} does not depend on the vertex $p$.

On the other hand, given $u\in \mS(P,\nu)$, the norm of any derivative $|\nabla^{g_u}\nabla^{g_u}\cdots \nabla^{g_u} \psi|_{g_u}$ is a smooth function on $\overline{P}$ as soon as $\psi$ is a smooth $\kt$--invariant function on $\overline{P}\times \kt$. Indeed, $|\nabla^{g_u}\nabla^{g_u}\cdots \nabla^{g_u} \psi|_{g_u}$ is then a $T_p$--invariant function on $M_p$ for each vertex $p$ of $P$ and thus, a smooth function on the image of the moment map of $(M_p,\omega_p,T_p)$ which is $\cup_{F\in\mF_p(P)}\mathring{F}$, see~\S\ref{subSectCPT}. These smooth continuations (one for each vertex) coincide when overlapping and thus $|\nabla^{g_u}\nabla^{g_u}\cdots \nabla^{g_u} \psi|_{g_u}\in C^{\infty}(\overline{P})$.

The above comments provide a scope for extending standard norms on functional spaces on $\overline{P}$. Namely, given $u\in \mS(P,\nu)$ and $\vol$, we take the pointwise norms of the derivatives on $P\times \kt$ (or on $M_p$ if applicable) with respect to the K\"ahler metric $g_u$ while we integrate over $\overline{P}$ using the volume form $\vol$. Therefore, we define $L^p$--norms, $C^k$--norms, H\"older norms on suitable spaces of functions on $\overline{P}$ giving rise to the definition of $L^p$--space, $C^k$--space and Sobolev spaces. These spaces do not depend on $\vol$ and coincide respectively with their ($T$--invariant) namesake on toric K\"ahler orbifolds in the rational case. Moreover, even when $(P,\nu)$ is non rational, they behave as if they were defined on a K\"ahler compact manifold: Sobolev inequalities, H\"older inequalities, Schauder estimates (for smooth operators on $\overline{P}$), Kondrakov Theorem hold as well, see~\cite{JoyceBOOK}. Because \eqref{intP2} is \eqref{intP} in the rational 
case, a lot of proofs are formally the same.
\begin{rem} The boundary condition on symplectic potentials given by~\cite{H2FII}, recalled in~\S\ref{subSectBC}, implies that $H=(\mbox{Hess } u)^{-1}$ is the restriction of a smooth $S^2\kt^*$--valued function on $\overline{P}$. Thus, as an example of what has been said above, for $f\in C^{\infty}(\overline{P})$, $g_u(\nabla^{g_u}f,\nabla^{g_u}f)= \sum_{i,j=1}^nH_{ij} f_{,i}f_{,j} \in C^{\infty}(\overline{P})$ (with notation~\eqref{ActionAnglemetric}).\end{rem} 

\subsection{Maximum Principle}\label{subSECTcommentLOCAL}
The Laplace operator of $g_u$ when restricted to the space of $\kt$--invariant functions on $P\times \kt$ is \begin{equation} \label{restLAPLACIAN}
\Delta^{u}_| = \sum_{i,j=1}^n H_{ij}\frac{\del^2}{\del x_j\del x_i} + \frac{\del H_{ij}}{\del x_j}\frac{\del}{\del x_i}.\end{equation}
This is a smooth operator on $C^{\infty}(\overline{P})$ satisfying $$\int_P h\Delta^{u} f\vol = \int_P \sum_{i,j=1}^nH_{ij} f_{,i}h_{,j} \vol =: \langle df,dh \rangle_u$$ as ensured by the boundary conditions on $H$. In particular, $\Delta^{u}$ is a symmetric operator on $C^{2}(\overline{P})$ whose kernel consists in constant functions. Moreover, it is elliptic on $P$ (but not uniformly elliptic). Let see why the Maximum Principle holds in this context as well. If $L_1,\dots L_k$ denote the defining affine-linear funtions of $P$, the operator $\Delta^{u}$ is uniformly elliptic on $$P_{\epsilon} =\{x\in \kt^* \,|\, L_k(x)>\epsilon,\, k=1,\dots, d\}$$ for any $\epsilon >0$. The classical Maximum Principle tells us that a function $f \in C^{2}(\overline{P})$ satisfying $-\
\Delta^u f\geq 0$ on $P_\epsilon$ attains a maximum on the boundary of $P_\epsilon$. Passing to the limit, we get that $f$ must attain a maximum at a point, say $p$, of $\del P$ whenever $-\Delta^u f\geq 0$ on $P$. But if $q$ is a vertex of the face in which lies $p$ (or is $p$ itself if $p$ is a vertex) then $f \in C^{2}(M_q)^{T_q}$ and reaches a maximum at a point in $M_q$ while $-\Delta^u f\geq 0$ (as a function on $M_q$, $\Delta^{u}$ is the Laplacian for the K\"ahler metric $g_u$ on $M_q$). Then, $f$ is constant. In sum, we have the following lemma 
\begin{lemma}
 Let $f\in C^2(\overline{P})$ if there exists $u\in \mS(P,\nu)$ such that $\Delta^u f\leq 0$ on $P$ then $f$ is constant.
\end{lemma}

\begin{rem}
 Of course the same principle holds for any symetric operator which is elliptic on $P$ and corresponds to an elliptic operator on each chart $M_p$. 
\end{rem}

\begin{rem}\label{selfadjoint} Whenever $(P,\nu)$ is rationnal and to $u\in \mS(P,\nu) $ is associated to a toric compact K\"ahler manifold $(M,\omega, g_u T)$, $\Delta^{u}$ is self-adjoint on the Hilbert space where it is defined, that is $\mathcal{W}^{1,1}(M)$. Now the subset $\mathcal{W}^{1,1}(M)^T$ of $T$--invariant functions is a closed Hilbert subspace and thus, the restriction of $\Delta^{u}$ is still self-adjoint. This means that the operator $\Delta^{u}_|$, as written in~\eqref{restLAPLACIAN}, is self-adjoint on $\mathcal{W}^{1,1}(\overline{P})$. There is no reason for that to fail in the non rational case.
\end{rem}
\subsection{A group of cohomology}

A classical approach adopted in K\"ahler geometry of compact manifolds, is to fix a complex structure $J$ on a compact manifold $M$ and a K\"ahler class $\Omega\in H^{1,1}(M,\bR)$ and study the space of compatible K\"ahler structures $(g,\omega,J)$ with $\omega\in\Omega$. (This is equivalent to fix $\omega$ instead, by Moser's Theorem.) This approach makes sense in our setting as well even though the cohomology of $P\times\kt$ is trivial because of the following important fact, explained in~\cite{H2FII}, due to the combination of a result of Schwarz~\cite{schwarz} and the Slice theorem. 

\begin{lemma}\label{lemSchwarz} Let $(M,\omega,g,J,T)$ be a compact toric K\"ahler orbifold with moment polytope $P$. Then $C^{\infty}(M)= C^{\infty}(\overline{P})$.
\end{lemma}

As a consequence of Lemma~\ref{lemSchwarz}, we get a way to define the $(1,1)$ group of cohomology.  
\begin{lemma}\label{cohomology} Let $(M,\omega,g,J,T)$ be a compact toric K\"ahler orbifold. Two real closed $(1,1)$--forms $\beta,\beta'$ on $M$ corresponding respectively to potentials $f$, $h\in C^{\infty}(P)$ (i.e $\beta=dd^cf$, $ \beta'=dd^cf'$ on $\mathring{M}$) are cohomologous if and only if $f-h \in C^{\infty}(\overline{P})$.
\end{lemma}

For a given symplectic potential $u\in \mS(P,\nu)$, the potential of the Ricci form associated to $g_u$ has been computed in~\cite{BurnsGuil}, to be
\begin{equation}\label{RicciPotential}
R_u(x)= \frac{1}{2} \mathrm{log }\,\mathrm{det }\,(\mathrm{Hess }\,u)_{x}. 
\end{equation}

\begin{rem}\label{remMONOTONE}Thus, using Lemma~\ref{cohomology}, $(M,\omega)$ is monotone with constant $\lambda>0$ if and only if for any symplectic potential $u\in\mS(P,\nu)$, $R_u-\lambda \tilde{\phi}\in C^{\infty}(\overline{P})$ where $\tilde{\phi}(x)=\phi((du_o)_x)$. This condition makes sense in the non rational case as well and is equivalent to the fact that $(\Polytope,\nu)$ is monotone in the sense of Definition~\ref{defMONOTONE}, see~\cite{don:Large}.\end{rem}

\begin{lemma}\label{LEMconeMONOTONE} Let $\Polytope$ be a polytope of dimension $n$, there is a $(n+1)$--dimensional cone of normals $\nu$ for which $(\Polytope,\nu)$ is monotone. Moreover, this cone is parameterized by $\bR_{>0}\times P$ via the map $$(\lambda,p)\mapsto\left\{\frac{\lambda \nu_1}{L_1(p)},\dots, \frac{\lambda \nu_d}{L_d(p)}\right\}$$
where $\nu$ is any given set of normals for $\Polytope$ with defining function $L_1,\dots, L_d$.\end{lemma}

\subsection{Compactifiable forms} In the rational case, the $1$--forms $dx_1,\dots, dx_n$ are well-defined on the compact orbifold $M$ and thus a $k$--form $\psi \in\Omega^k(\overline P)$ is pulled-back to give a {\it basic} $k$--form on $M$. Here basic should be undetstood has $\kt$--basic, meaning that the contraction of $\psi$ and $d\psi$ by any element $X_v$ with $v\in \kt$ vanish identically. However, not every $\kt$--invariant $k$--form on $\overline{P}\times \kt$ corresponds to a form that is the restriction of a smooth form on $M$. This is the case for $\psi\in \Omega^k(\overline{P}\times \kt)$ if, for any $k\in\{1,\dots, d\}$, near $\mathring{F}_k$ the contraction of $\psi$ with $\frac{1}{L_k}X_{\nu_k}$ is smooth on $P\cup \mathring{F}_k$. To see this, we consider a chart $M_p$ for a vertex $p\in F_k$ and observe that, in polar coordinates, $|z_i|^2d\theta_i$ is smooth. We call these forms, those who behave as if they were defined on a compact orbifold, {\it compactifiable forms}. Let $\psi$ be such 
a $\kt$--invariant $(2n-1)$--form. Thanks to invariance $d\psi = \sum_{i=1}^n \frac{\del}{\del x_i} \psi_i dx\wedge d\theta$ and $$\int_{P} (d\psi)_{|_P} = \int_{P} d  \left(\sum_{i=1}^n (-1)^{i+1} \psi_i dx_1\wedge\dots \wedge \widehat{dx_i} \wedge dx_n\right) = \int_{\partial P} \widehat{\psi}$$ where $\widehat{\psi} = \sum_{i=1}^n (-1)^{i+1} \psi_i dx_1\wedge\dots \wedge \widehat{dx_i} \wedge dx_n$. But the condition that the contraction of $\psi$ with $\frac{1}{L_k}X_{\nu_k}$ is smooth on $P\cup \mathring{F}_k$ implies that $\widehat{\psi}$ vanishes on $\partial P$. 


\section{Some classical theorems}\label{secKRSorbi}
\subsection{Yau's Theorem}

\begin{theorem}[\cite{calabiConj,Y}] Given a compact complex manifold $(M,J)$ of K\"ahler type and a K\"ahler class $\Omega$, for each $(1,1)$--form $\rho\in 2\pi c_1(M)$ there exists a unique K\"ahler form $\omega\in \Omega$ such that $\rho$ is the Ricci form of the K\"ahler structure $(\omega, J)$. 
\end{theorem}

Since the Ricci form only depends on the volume form, Yau's Theorem reads as: given a smooth function $F\in C^{\infty}(M)$, satisfying $\int_M e^F\omega_o^n =\int_M\omega_o^n$ there exists a unique $\psi \in C^{\infty}(M)$ such that $\int_M\psi \,\omega^n=0$, $\omega_o +dd^c \psi >0$ and  
 \begin{equation}\label{MAY}
   (\omega_o +dd^c \psi)^n =e^F\omega_o^n.
 \end{equation} This version furnishes a pde. In \cite{calabiConj}, Calabi showed the uniqueness of solution of this pde and suggested a continuity method to prove the existence. He proved the openness of the set of solutions of the Monge--Amp\`ere equation~\eqref{MAY} while Yau~\cite{Y} produced a priori estimates. Good references for this proof are also \cite{JoyceBOOK, ast58}.\\
 
 Let $(P,\nu)$ be a labelled polytope with $P\subset \kt^*$. We fix a symplectic potential $u_o\in\mS(P,\nu)$, thus the K\"ahler structure $(g_o,\omega_o, J)$ on $\kt\oplus i\kt$ ($J$ denote the endomorphism of the tangent bundle induced by $i$), we denote $\rho_o$ its Ricci form and $\phi_o$ the Legendre transform of $u_o$. Thanks to Lemma~\ref{cohomology} and~\eqref{RicciPotential}, Yau's Theorem reads in our setting as: 
\begin{theorem} Given $R\in C^{\infty} (P)$ such that $ R- R_{u_o} \in C^{\infty}(\overline{P})$, there exists $u\in \mS(P,\nu)$ such that $R=\frac{1}{2} \mathrm{log }\,\mathrm{det }\,(\mathrm{Hess }\,u)_{x}$ and this solution is unique up to addition by an affine-linear function.\end{theorem}
 
 With the discussion of Section~\ref{sectGEOMnonRAT}, we should be convince that the whole proof of Calabi and Yau holds in this setting as well, but let see some details. First, we pass to the complex side of the picture in order to rely on the existing litterature on this topic.
 
%
  
 The path of equations considered is the one obtained by taking $e^{F_t} = \frac{e^{tF}}{\frac{1}{V_{\Omega}}\int_M e^{tF}\vol}$, where $V_{\Omega}= \int_M\vol$. The set $S$ of $t\in [0,1]$ such that there exists a solution of equation~\eqref{MAY} with $F=F_t$ is non empty, since at $t=0$ $\psi\equiv 0$ is a solution.    
 
The unicity of solutions is a consequence of the Maximum Principle. More precisely, we copy the explanation of~\cite{pg}, writting $\Psi$ the Hermitian endomorphism of the tangent defined by $\omega(X, \Psi(Y))= dd^c\psi(X,Y)$ and $\lambda_1,\dots, \lambda_n$ its real eigenvalues. Equation~\eqref{MAY} is $$\prod_{i=1}^n(1+\lambda_i) =1+\sigma_1(\lambda_1,\dots, \lambda_n)+ \dots +\sigma_n(\lambda_1,\dots, \lambda_n)= e^F$$ where $\sigma_i$ is the $i$-th symmetric elementary functions. In particular,  $\omega_o +dd^c \psi >0$ if and only if $1+\lambda_i >0$ for all the $i$'s. Now, since
$$(\prod_{i=1}^n(1+\lambda_i))^{\frac{1}{n}} \leq \frac{1}{n}\sum_{i=1}^n ((1+\lambda_i))= 1 + \sigma_1(\lambda_1,\dots, \lambda_n)$$ 
and that $\sigma_1=-\Delta^{g_o} \psi$. We get \begin{equation}\label{unicityMAY}\Delta^{g_o} \psi \leq n(1- e^{\frac{F}{n}}). \end{equation} We conclude, using~\S\ref{subSECTcommentLOCAL}, that in our generalized setting as well if there is a solution of equation~\eqref{MAY} this solution is unique.   

Moreover, the linearisation of equation \eqref{MAY} is $$\dot{F}\mapsto \Delta^u \dot{F}$$ as explained in \cite{pg,JoyceBOOK,ast58}. Hence, the fact that $S$ is open follows the fact that the Laplacian defines an isomorphism of $C_{0,g}^{\infty}(M)=\{f\in C^{\infty}(M)\,|\, \int_M f \vol_g=0\}$. Of course the space of invariant functions $$C_{0,g}^{\infty}(M)^T=\{f\in C^{\infty}(M)^T\,|\, \int_M f \vol_g=0\}=C_{0,\vol}^{\infty}(\overline{P})$$ is closed under this isomorphism and the argument holds in our setting for the Laplacian~\eqref{restLAPLACIAN}, see~\S\ref{sectGEOMnonRAT}.

To prove that $S$ is a closed subset of $[0,1]$, one needs a priori estimates on solutions of the Monge--Amp\`ere equation~\eqref{MAY}. The estimates were found by Yau in~\cite{Y}. This is a great piece of work that inspired a lot of subsequent studies. In particular, in~\cite{TZ}, Tian and Zhu used and adapted Yau's ideas for a more complicated equation corresponding to K\"ahler--Ricci solitons. In~\S\ref{sectWZtheo}, we discuss how their study carries in our setting. The uniform bound (the $C^0$ estimate) follows a boot strapping method as explained in~\cite{JoyceBOOK} and uses only Stoke's Theorem, Sobolev embedding Theorem... that hold in our setting~\S\ref{sectGEOMnonRAT}.     
\subsection{Zhu's Theorem}\label{sectZhuTheo}
In~\cite{Z}, Zhu considered the following problem : on a compact Fano manifold $(M,J)$, given a K\"ahler form $\tau \in 2 \pi c_1(M)$ and a holomorphic vector field $Z$ on $M$, does there exist a K\"ahler form $(\omega,g)$ such that  
\begin{equation}\label{ZhuKRsoltyped} \rho^g- \tau =\mL_Z\omega\end{equation} and, if it exists, is it unique ? Zhu proved unicity of solution (up to automorphism) and exhibited necessary and suffisant conditions on the vector field $Z$ for a solution to exist. In the toric case, choosing an $T$--invariant form $\tau$ these conditions are fulfilled in the toric case whenever $Z=JX_a-iX_a$ for any $a\in \kt$. Moreover, in this case, the solution is $T$--invariant. Via Lemma~\ref{cohomology}, with the same notation as before (picking a reference point $\omega_o$...) Zhu's result reads as follows 
\begin{theorem} Given a convex function $R\in C^{\infty} (P)$ such that $R- R_{u_o} \in C^{\infty}(\overline{P})$ and $a\in \kt$, there exists $u\in \mS(P,\nu)$ such that $R_u - R= a$ and this solution is unique up to addition by an affine-linear function.\end{theorem}
 Zhu actually worked on the following version of the equation~\eqref{ZhuKRsoltyped}, a solution $\omega_\psi= \omega_o+dd^c\psi$ satisfies     
  \begin{equation}\label{MAYcooZ}\begin{split}
   \det(g_{\i\bar\j}+\psi_{\i\bar\j})= e^{f_o-\theta_Z - Z.\psi}\det(g_{\i\bar\j}),\\
   g_{\i\bar\j}+\psi_{\i\bar\j} >0
   \end{split}
 \end{equation} where $\theta_Z,f_o\in C^{\infty}(\overline{P})$ satisfies $\mL_Z\omega_o=dd^c\theta_Z$ and $\rho_o -\tau = dd^cf_o$. He observed that $Z.\psi$ needs to be a real-valued function. In our case, working with $\kt$--invariant functions and holomorphic vector field induced from $\kt\oplus i\kt$ this condition is satisfied. Then using a continuity method Zhu proved uniqueness and existence of solutions assuming that $\tau$ is positive definite $(1,1)$--form. The path of equations he considered starts at the pde corresponding to Yau's result. So we get the non-emptyness thanks to Yau's Theorem (which holds in our context by the last section). The proof of openness and uniqueness use standard arguments and facts on compact manifold (Stokes, integration by parts, Maximum Principle...). A key ingredient for estimates leading to the closeness part of Zhu's proof is an a priori bound on $|Z.\psi|$ on compact K\"ahler manifolds as soon as $(\im Z).\psi=0$ and $\omega_o +dd^c \psi >0$. The proof of Zhu use classification of complex surfaces and cannot be directly adapted to our generalized setting. However, we only need a weaker result: since $\psi(t)=\phi(t)-\phi_o(t)$ and $Z= JX_a-iX_a$ for some $a\in \kt$, with respect to the coordinates $t+i\theta$, we have
 \begin{equation}\label{insteadZHU}\begin{split}
  |(Z.\psi)_t|= |d\psi(JX_a-iX_a)_t|&=|\sum_{i=1}^n a_i \left(\frac{\del\phi}{\del t_i} - \frac{\del\phi_o}{\del t_i}\right)| =|\langle a, x\rangle - \langle a, x_o\rangle|\\
  & \leq \max\{\langle a, x \rangle\,|\, x\in \overline{P}\} -\min\{\langle a, x_o\rangle\,|\, x_o \in \overline{P}\}.
\end{split}
   \end{equation} This gives the desired bound. For the rest of the estimates, Zhu adapts Yau's argument and we will see a more complicated version in the next section.        

\subsection{The Theorem of Wang and Zhu}\label{sectWZtheo}

\begin{theorem}[\cite{WZ}] Given a compact Fano toric manifold $(M,J,T)$ and $Z_a$ the K\"ahler-Ricci vector field (where $a\in \kt$ is defined by~\eqref{RSvector}). For any $\lambda>0$, there exists a unique $T$--invariant K\"ahler form $\omega\in 2\pi \lambda c_1(M)$ such that $(g,\omega,J)$ is a K\"ahler--Ricci soliton with respect to $Z_a$. 
\end{theorem}

Let $(P,\nu)$ be a monotone labelled polytope with preferred point $p\in P$ and $a\in\kt$ be defined by~\eqref{RSvector}. Again, we fix a symplectic potential $u_o\in\mS(P,\nu)$ and use the same notation as before.


If $(g,\omega,J)$ is a K\"ahler--Ricci soliton with respect to $Z= JX_a-iX_a$ in the sense that it satisfies~\eqref{KRSequation} with $\lambda>0$, one can write
 \begin{equation}\label{KRSequation2}
  -dd^c\log \frac{\omega^n}{\omega_o^n} =dd^c(\theta_Z-f_o +\lambda\psi + Z.\psi).
\end{equation} Indeed, this is equation~(\ref{KRSequation}) with $\rho -\rho_o=-dd^c\log \frac{\omega^n}{\omega_o^n}$ and $n\lambda\int_{ \Polytope}\vol= \int_{\del \Polytope}\length_{\nu}.$ Wang and Zhu used a continuity method on the modified equation 
  \begin{equation}\label{MA}
(\omega_o +dd^c \psi)^n =e^{f_o-\theta_Z - s\psi -Z.\psi}\omega_o^n
\end{equation} for a parameter $s\in [0,1]$. The normalization imposed is \begin{equation}\label{normalization}\int_{M}e^{f_o}\omega_o^n=\int_{M}\omega_o^n,\;\;\; \;\;\;\int_{M}e^{\theta_Z+Z.\psi}\omega^n=\int_{M}e^{f_o -\psi}\omega_o^n=\int_{M}\omega_o^n.\end{equation}
For $s=0$ the existence and uniqueness of the solution is due to Zhu's result~\cite{Z}, which holds in our context see~\S\ref{sectZhuTheo}. Again the proof of openness and uniqueness use standard stuff on compact manifold (Stokes, integration by parts, Maximum Principle...) that hold in our setting as explained in Section~\ref{sectGEOMnonRAT}. The uniform estimate follows some nice arguments of convex Euclidean geometry. We will explained in more details why their higher estimates hold in our generalized setting, precisely
\begin{lemma}\label{uptoC3estimates} Fix $0\leq s\leq 1$. If $\psi =\phi-\phi_o\in C^{\infty}(\kt)$ is a solution of
  \begin{equation}\label{MA}
(\omega_o +dd^c \psi)^n =e^{f_o-\theta_Z - s\psi -Z.\psi}\omega_o^n
\end{equation} where $\phi$, $\phi_o$ are the Legendre transforms of $u$, $u_o\in\mS(P,\nu)$ and $\phi_o$ is a potential for $\omega_o$ then a $C^0$ bound on $\psi$ provides $C^2$ and $C^3$ bounds on $\psi$.
\end{lemma}

%

\begin{rem} Note that the normalization \eqref{normalization} only affects $\psi$ up to an additive constant, so the condition $\im (t\mapsto(d\phi)_t)= \im (t\mapsto(d\phi_o)_t)=P$ is not over determined.
\end{rem}

\begin{proof}[Proof of Lemma~\ref{uptoC3estimates}]

Equation~(\ref{MA}) is \eqref{MAY} with $F$ replaced by $f_o-\theta_Z- s\psi -Z.\psi$. Thus Tian--Zhu~\cite{TZ} had to adapt Yau's approach. Recall that Zhu gave an a priori bound on $|Z.\psi|$ that still holds in our context \S\ref{sectZhuTheo}. Apart from this fact, the arguments of Tian--Zhu are essentially local, using the compactness of the manifold only to get bounds on various continuous functions (depending on $(\omega_o, g_{u_o})$) appearing in the equations. Applying the principle explained in~\S\ref{subSECTcomment} is then enough to claim that the estimates hold in our setting.\\

We present below the details for getting the second order estimate.\\

As a first step, a local computation shows that, for solutions of \eqref{MAY}, a priori bounds on $|F|$ and $|\Delta^{g_o} \psi|$ give a priori bounds on $|dd^c\psi|_{g_o}$, see \cite[Proposition 5.3.4]{JoyceBOOK}. Hence, for solutions of~(\ref{MA}), it is sufficient to bound $|\Delta^{g_o} \psi|$ and $|s\psi + Z.\psi|$. A bound on $|Z.\psi|$ follows from \eqref{insteadZHU}. It remains to find a bound on $|\Delta^{g_o} \psi|$. We only have to find an upper bound to $\Delta^{g_o} \psi$ since $$0<\mbox{tr}_{\omega_o}(\omega_o+ dd^c \psi) = n+ \Delta^{g_o} \psi$$ where $\mbox{tr}_{\omega_o}$ is the trace with respect to $\omega_o$. Tian and Zhu computed that
 \begin{equation}\label{estiLOCAL1}\begin{split}
   \Delta^{g}((n+\Delta \psi)&\exp(-c\psi)) \geq \exp(-c\psi)(c+\inf_{\i\neq l}R_{\i\bar{\i}l\bar{l}})(n+\Delta\psi)\left(\sum_{\i}\frac{1}{1+\psi_{\i\bar{\i}}}\right)\\
   &+\exp(-c\psi)\left((\Delta(f_o-\theta_Z- \lambda\psi -Z.\psi) -n^2\inf_{\i\neq l}R_{\i\bar{\i}l\bar{l}}) -cn(n+\Delta\psi)\right)
 \end{split}
 \end{equation} at any point $p$, where $\Delta=\Delta^{g_o}$, $R_{\i\bar{\j}k\bar{l}}$ are components of the curvature tensor of the metric $g_o$ with respect to holomorphic coordinates, say $z$, chosen at $p$ so that $(g_o)_{\i\bar{\j}}=\delta_{\i\j}$ and $\psi_{\i\bar{\j}}=\delta_{\i\j}\psi_{\i\bar{\i}}$ (this convention is used in local computation mentioned above).

 Note that each function appearing in the right hand side of~\eqref{estiLOCAL1} is smooth on $\overline{\Polytope}$. In particular, $x\mapsto R_{\i\bar{\i}l\bar{l}}(t_x)$ defines a smooth function on $\overline{P}$ as one can see easily using the boundary condition on symplectic potentials as they are stated in Proposition~\ref{propHcondbord}. Indeed, changing the variables from $t$ to $x$, one gets

 \begin{equation}\label{estiLOCAL2}\begin{split}
R_{\i\bar{\i}l\bar{l}} &=- \frac{\del^2 (g_{o})_{\i\bar{\i}}}{\del z_l\del\bar{z}_{l}} +\sum_{p,q}g_o^{p\bar{q}}\frac{\del (g_{o})_{p\bar{\i}}}{\del z_l}\frac{\del (g_{o})_{i\bar{q}}}{\del\bar{z}_l} \\
&= -\sum_{r,s} H_{ls}\frac{\del }{\del x_s}\left(H_{lr}\frac{\del H_{ii}}{\del x_r}\right) + \sum_{r,s} H_{rs}\frac{\del H_{il}}{\del x_s} \frac{\del H_{il}}{\del x_r}.
\end{split}
 \end{equation}
Now, let $p\in \overline{P}$ be a point where the function $\exp(-c\psi)(n+\Delta \psi)$ attains its maximum. Then, using \eqref{insteadZHU} and the compactness of $\overline{P}$, we can show, as Tian and Zhu did, that at this point $p$, there exist $C_1,C_2>0$ such that
 \begin{equation}\label{estiLOCAL3}
   \Delta(-f_o+\theta_Z+\lambda\psi +Z.\psi)\leq C_1+C_2(n+\Delta\psi).
 \end{equation} Inserting this into \eqref{estiLOCAL1} and using some local formulas (following~\cite{Y}), there are constants  $C_3,C_4,C_5$ independent of $\psi$ such that
 \begin{equation}\label{estiLOCAL4}\begin{split}
  \Delta^{g}((n+\Delta \psi)\exp(-c\psi)) \geq  -&\exp(-c\psi) (C_3 + C_4(n+\Delta\psi)) \\
  &+C_5\exp(-c\psi +\frac{s}{n-1}\psi)(n+\Delta\psi)^{n/(n-1)}.
  \end{split}
 \end{equation} Then, Yau applied the Maximum Principle: {\it the left hand side of \eqref{estiLOCAL4} is the Laplacian of a function at its maximum so it must be negative}. This argument holds if the maximum is not attained on the boundary of a manifold (which was obviously the case in Yau and Tian--Zhu setting). Actually, it works in our setting as well: if $p\in F \subset\del P$ where $F$ is a face containing a vertex, say $q$, then the left hand side of \eqref{estiLOCAL4} is the Laplacian (of a smooth metric $g_u$) of a function defined on $M_q$ (see~\S\ref{subSectBC}) which attains a local maximum at $p$ so it must be negative. Hence, we get the Tian--Zhu estimate in our generalized setting :
  \begin{equation}\label{estiFIN}
  (n+\Delta\psi_s) \leq C (1+ \exp(-s \inf_{M} \psi_s))\exp(-c(\psi_s + \inf_{M}\psi_s))
 \end{equation} for constants $c,C$ independent of $\psi$.

Following the same argument, the $C^3$--estimate of Tian and Zhu holds as well.\end{proof}
\section{Proof of Theorem~\ref{theo1}}
The proof of Theorem~\ref{theo1} relies on the following lemma.
\begin{lemma}\label{lemSIMPLEX} Let $(\Polytope, \nu)$ be a labelled polytope and let $\vol$ be a volume form on $\kt^*$. The linear map $\Psi:{\rm Aff}(\Polytope,\bR)\longrightarrow \kt$ defined as
 \begin{equation}\label{defnPHI}\Psi_{(\Polytope,\vol)}(f)= \sum_{k=1}^d \left(\int_{F_k}f \length_{\nu}\right)\nu_k \end{equation}
does not depend on the set of normals $\nu$ and $\Psi_{(\Polytope,\vol)}(f) =0$ as soon as $f$ is constant. Moreover, seen as an endomorphism of $\kt$,
\begin{equation}\label{explem} \Psi_{(\Polytope,\vol)}=-{\rm vol}(\Polytope,\vol)\,{\rm Id}
 \end{equation}
\end{lemma}

\begin{proof} The first claim (no dependence on $\nu$) is straightforward. Suppose that $\overline{\Polytope} = \overline{\Polytope}'\cup \overline{\Polytope}''$ such that $F=\overline{\Polytope}'\cap\overline{\Polytope}''$ is a facet in both $\Polytope'$ and $\Polytope''$. Note that $\Polytope'$ and $\Polytope''$ induce opposite orientations on $F$. If we choose $v\in\kt$ a normal vector to $F$ inward to $\Polytope'$ and denote $\length'$ the form on $F$ such that $v\wedge \length'= -\vol$ then $-v\in\kt$ is inward to $\Polytope''$ and  $-v\wedge -\length'= -\vol$. Therefore,
 $$\Psi_{(\Polytope,\vol)}(f)= \Psi_{(\Polytope',\vol)}(f)+\Psi_{(\Polytope'',\vol)}(f).$$
 By using a triangulation, that is, $n$--simplices $\Polytope_{1},\dots, \Polytope_{N}$ such that $ \overline{\Polytope}=\bigcup_{\alpha=1}^N \overline{\Polytope}_{\alpha}$ and $\Polytope_{\alpha}\cap\Polytope_{\beta}=\emptyset$ if $\alpha\neq \beta$ we get \begin{equation}\label{triangul}\Psi_{(\Polytope,\vol)}(f)= \sum_{\alpha=1}^N\Psi_{(\Polytope_{\alpha},\vol)}(f).\end{equation}

Consider the simplex $\Sigma=\{x\in\bR^n\;|\; x_i> 0, \sum_{i=1}^n x_i < 1 \}$ together with the set of normals $e=\{e_1,\dots, e_n, e_0 = -\sum_{i=1}^n e_i \}$. Let $f\in{\rm Aff}(\Sigma,\bR)$, we have
\begin{equation}\label{expSIMPLEX} \Psi_{(\Sigma,dx_1\wedge \dots \wedge dx_n)}(f)= \sum_{i=1}^n  \left( \int_{E_i}f \length -\int_{E_0}f \length \right) e_i
 \end{equation} where for $k=0,\dots, n$, $E_k$ denotes the facet normal to $e_k$.\\

Thus, $\int_{E_i}\length=\int_{E_0}\length = \frac{1}{(n-1)!}$ implies that $\Psi_{(\Sigma,dx_1\wedge \dots \wedge dx_n)}(f)=0$ as soon as $f$ is constant.

On the other hand, for each $i=1,\dots, n$, the invertible affine-linear map
\begin{align*}
\psi_i\, :\; \bR^n &\longrightarrow \bR^n \\
x &\longmapsto (x_1,\dots, x_{i-1}, 1-\sum_{j=1}^n x_j,x_{i+1}, \dots,x_n )
\end{align*}
reverses the orientation of $\bR^n$, sends $E_i$ to $E_0$ and satisfies $\psi_i^*e_0=e_i$. It follows that $\psi_i^*\length_{|_{E_0}}=-\length_{|_{E_i}}$ and in particular $\int_{E_0}f\length = -\int_{E_i}\psi_i^*(f\length)= \int_{E_i}(f\circ\psi_i ) \length$. Thus, by writing $f$ in coordinates $f(x)=f_0+\sum_{j=1}^n f_jx_j$, equation~(\ref{expSIMPLEX}) becomes
\begin{equation*}\begin{split}
\Psi_{(\Sigma,dx_1\wedge \dots \wedge dx_n)}(f) & = \sum_{i,j=1, j\neq i}^n f_j \left( \int_{E_i}x_j \length -\int_{E_0}x_j \length \right) e_i -\sum_{i=1}^n f_ie_i\int_{E_0}x_i \length \\
&=-\sum_{i=1}^n f_ie_i\int_{E_0}x_i \length.
\end{split}
 \end{equation*}
 Observe that for any $i=1,\dots, n$ the $(n-1)$--form $(-1)^{i+1}x_i \,dx_1\wedge \dots \wedge\widehat{dx_i}\wedge \dots \wedge dx_n$ vanishes identically when restricted to the facets $E_1,\dots,E_n$ and coincides with $x_i\length$ on $E_0$. Hence, we have \begin{equation*} \int_{E_0}x_i \length =(-1)^{i+1}\int_{\Sigma}d \,(x_i \,dx_1\wedge \dots \wedge\widehat{dx_i}\wedge \dots \wedge dx_n) =\int_{\Sigma}dx_1\wedge \dots \wedge dx_n
 \end{equation*}
 and, seen as an endomorphism of $\bR^n$, $\Psi_{(\Sigma,dx_1\wedge \dots \wedge dx_n)} = -\mbox{vol}(\Sigma,dx_1\wedge \dots \wedge dx_n)\mbox{Id}.$

There is only one class of affinely equivalent $n$--simplices: for any $n$-simplex $\Polytope_{\alpha} \subset\kt^*$ there is an affine-linear map $\phi_{\alpha}: \kt^*\ra \bR^n$ such that $\phi_{\alpha}(\Polytope_{\alpha})= \Sigma$. Hence, \begin{equation}\label{expSIMPLEX3} \begin{split} \Psi_{(\Polytope_{\alpha},\vol)} &=\phi_{\alpha}^*\circ\Psi_{(\Sigma,(\phi_{\alpha}^{-1})^*\vol)}\circ (\phi^{-1})^*\\
 & = -\frac{(\phi_{\alpha}^{-1})^*\vol}{dx_1\wedge \dots \wedge dx_n} \mbox{vol}(\Sigma,dx_1\wedge \dots \wedge dx_n)\mbox{Id}\\
 & = -\mbox{vol}(\Polytope_{\alpha},\vol)\mbox{Id}. \end{split}\end{equation}
The lemma follows then from \eqref{triangul}. \end{proof}

Let $\Polytope$ be a $n$--dimensional polytope in a vector space $\kt^*$ of dimension $n$ and denote by $F_1,\dots,F_d$ its facets. Up to translation, one can assume that $\Polytope$ contains the origin and denote $\nu$ the (unique) set of inward normals for which the defining functions of $(P,\nu)$ satisfies $L_1(0)=\dots=L_d(0)=1$. To prove Theorem~\ref{theo1}, we have to show that there exists only one point $p\in P $ such that $\left(\Polytope, \left\{ \frac{\nu_1}{L_1(p)},\dots, \frac{\nu_d}{L_d(p)}\right\}\right)$ has a constant extremal affine function, see Lemma~\ref{LEMconeMONOTONE}.\\

Notice that any set of normals $\nu(r)$ on $\Polytope$ corresponds to a $r=(r_1,\dots,r_d)\in\bR_{>0}^d$ via the variation of normals: $$\nu(r)= \left\{ \frac{1}{r_1}\nu_1,\dots, \frac{1}{r_d}\nu_d \right\}.$$

Choose a basis $e_1,\dots,e_n$ of $\kt$ and corresponding coordinates $(x_1,\dots,x_n)$ on $\kt^*$. By considering the linear system~(\ref{systCHPext}), we know that, for a given set of normals $\nu(r)$, $(\Polytope,\nu(r))$ has a constant extremal affine function if and only if $$\zeta_{(\Polytope,\nu(r))}=\zeta_0 = \frac{Z_0(\Polytope,\nu(r))}{W_{00}(\Polytope)}$$ which happens if and only if
\begin{equation}\label{systEQbary}  W_{i0}(\Polytope)Z_0(\Polytope,\nu(r)) = W_{00}(\Polytope)Z_i(\Polytope,\nu(r)) \;\;\;\;\;\; i= 1,\dots, n. \end{equation}
Since $(\length_{\nu(r)})_{|_{F_k}}= r_k(\length_{\nu})_{|_{F_k}}$, the system of equations~(\ref{systEQbary}) is linear in $r$ and reads,
\begin{equation}\label{systEQbary0} W_{i0}(\Polytope) \left(\sum_{l} \int_{F_l} \length_{\nu} r_l\right)=  W_{00}(\Polytope) \sum_{l} \int_{F_l} \mu_i\length_{\nu} r_l \;\;\;\;\;\; i= 1,\dots, n.\end{equation}

\noindent {\bf Notation}: Indices $i$ and $j$ run from $1$ to $n$ while indices $k$ and $l$ run from $1$ to $d$.\\

Observe that, since $\langle \nu_k,\mu\rangle =-1$ on $F_k$, the vector field $X=\sum_{i=1}^n \mu_i\frac{\del}{\del \mu_i}$ satisfies
 $$\iota_X\vol = \length_{\nu}  \;\;\;\;\;\; {\rm div } X = n  \;\;\;\mbox{ and }\;\;\; {\rm div } \mu_jX = (n+1)\mu_j.$$ In particular, $W_{00} = \frac{1}{n} \sum_{k} \int_{F_k} \length_{\nu}$, $W_{i0} = \frac{1}{n+1} \sum_{k} \int_{F_k} \mu_i\length_{\nu}$ and the system of equations~(\ref{systEQbary}) for $(\Polytope,\nu(r))$ becomes,  $i= 1,\dots, n$,
\begin{equation}\label{systEQbary3} \sum_{k,l} \left(n\int_{F_k} \mu_i\length_{\nu} \int_{F_l} \length_{\nu} - (n+1) \int_{F_k} \length_{\nu} \int_{F_l} \mu_i\length_{\nu}\right)r_l =0. \end{equation}

Since we seek monotone polytopes, by Lemma~\ref{LEMconeMONOTONE}, we can restrict our attention to those $r=(r_1,\dots,r_d)\in\bR_{>0}^d$ such that there exists $p\in\mathring{M}$ with $$r =(L_1(p),\dots, L_d(p)).$$ The system of equations~(\ref{systEQbary3}) becomes, $i= 1,\dots, n$,
\begin{equation}\label{systEQbary4}  \sum_{k,l} \left(n\int_{F_k} \mu_i\length_{\nu} \int_{F_l} \length_{\nu} - (n+1)\int_{F_k} \length_{\nu}\int_{F_l} \mu_i\length_{\nu}\right)(\langle \nu_{l},p\rangle + 1) =0.
\end{equation}

Lemma~\ref{lemSIMPLEX} implies that $\sum_{k=1}^d \int_{F_k} \length_{\nu} \nu_k=0$, thus equation~(\ref{systEQbary4}) reads
\begin{equation}\label{systEQbary5}
 \sum_{l}  \int_{F_l} \mu_i\length_{\nu} \langle \nu_{l},p\rangle  = \frac{-1}{n+1} \sum_{l}  \int_{F_l} \mu_i\length_{\nu}.
\end{equation}
The left hand side is just $\langle\Psi_{(\Polytope,\vol)}(e_i^*),p\rangle$ where $\Psi_{(\Polytope,\vol)} =-\mbox{vol}(\Polytope,\vol)\mbox{Id}$ thanks to Lemma~\ref{lemSIMPLEX}. This ensures that there is a unique solution $p\in\kt^*$ of the linear system~\eqref{systEQbary5} given as

\begin{equation}\label{solutionMONBARYC}\begin{split}
 p = &\frac{1}{(n+1)\mbox{vol}(\Polytope,\vol)}\left( \sum_{l}  \int_{F_l} \mu_1\length_{\nu},\dots,\sum_{l}  \int_{F_l} \mu_n\length_{\nu} \right)\\
 = & \frac{n}{(n+1)\int_{\del\Polytope}\length_{\nu}}\left(\int_{\del\Polytope} \mu_1\length_{\nu},\dots,\int_{\del\Polytope} \mu_n\length_{\nu} \right).
 \end{split}
\end{equation}
This solution lies in a segment between the origin and the center of mass of $(\del \Polytope,\length_{\nu})$ and thus lies in $\mathring{\Polytope}$. In conclusion, the labelled polytope $(P,\tilde{\nu})$ with $\tilde{\nu}_i=\frac{\nu_i}{L_i(p)}$ where $p$ is the barycenter of $(\Polytope,\vol)$  is monotone and has a constant extremal affine function.  This prove Theorem~\ref{theo1}.

\begin{corollary}\label{centerMASS=preferredpoint} The labelled polytope $(P,\nu)$ is monotone with preferred point $p_{\nu}$ and has a constant extremal affine function $A_{(P,\nu)}$ if and only if $$p_{\nu}= \frac{1}{\int_P\vol}\left(\int_Px_1\vol,\dots, \int_Px_n\vol\right)=\frac{1}{\int_{\del P}\length}\left(\int_{\del P}x_1\length,\dots, \int_{\del P}x_n\length\right).$$
\end{corollary}

Another simple corollary of Theorem~\ref{theo1} is that the linear space of $r\in\bR^d$ such that $A_{(P,\nu(r))}$ is constant meets the interior of the positive quadrant of $\bR^d$ and thus:

\begin{corollary}\label{coroSETcscK}
 Given a polytope $P$, there is a cone of dimension $d-n$ of inward normals $\nu$ such that the extremal affine function $A_{(P,\nu)}$ is constant. \end{corollary}
\section{Applications}\label{SECTapplication}
\subsection{Extremal K\"ahler equation}\label{openNORMALS}

Consider the extremal K\"ahler equation
\begin{equation}\label{extremalKeq}
S(u)= A_{(P,\nu)}
\end{equation}
for $u\in \mS(P,\nu)$ where $S(u)$ is the scalar curvature of $g_u$, given by Abreu's formula~\eqref{abreuForm}. If $u$ is solution of \eqref{extremalKeq}, then $g_u$ is an extremal K\"ahler metric in the sense of Calabi~\cite{calabi}. In~\cite{don:extMcond}, Donaldson proved\footnote{The argument is stated for $n=2$ but holds in general.} that the cokernel of the linearisation of the map $u\mapsto S(u)$ in $C^{\infty}(\overline{P})$ is the set of affine linear functions on $\overline{P}$. In particular, denoting $\bfN(P)$ the cone of normals inward to $P$, the linearisation of the extension of the scalar curvature map

\begin{align}
 \begin{split}
       S: \bigcup_{\nu\in \bfN(P)} \mS(P,\nu) &\longrightarrow
C^{\infty}(\overline{P})\\
           u &\longmapsto S(u)
 \end{split}
\end{align} is surjective on $\{ A_{(P,\nu)}\,|\,\nu\in \bfN(P)\} \oplus \big( \quot{C^{\infty}(\overline{P})}{\mbox{Aff}(P,\bR)}\big)$. Thus, together with the fact that $\bigsqcup_{\nu\in \bfN(P)}\mS(P,\nu)$ is path connected, the set $\bfE(P)$ of inward normals $\nu$ for which there exists an extremal K\"ahler metric $g_u$ with $u\in\mS(P,\nu)$ is open in $\bfN(P)$. Gathering this with Theorem \ref{theo1}, Theorem~\ref{WZtheoGLP} and Corollary~\ref{coroSETcscK}, we get Corollary~\ref{propEXISTENCEextremal}.

\subsection{Lattice polytopes and K\"ahler--Einstein orbifolds}\label{samePOLYTOPEs}

\begin{lemma}\label{lemmaRATIONALlattice} Let $(P,\nu)$ be a monotone labelled polytope with a constant extremal affine function. Then, $\mbox{span}_{\bZ}\{\nu\}$ is a lattice if and only if $P$ is a lattice polytope.
\end{lemma}
\begin{proof} One direction is straightforward : if $\nu$ spans a lattice $\Lambda$ there is a toric K\"ahler--Einstein orbifold with moment polytope $P$ and the general theory tells us that $P$ is a lattice polytope.

Conversely, first observe that the assumption implies that there is a lattice containing both the set of vertices of $P$
and the preferred point $p_{\nu}$. To see this, one can use Corollary~\ref{centerMASS=preferredpoint} once $\kt^*$ identified
with $\bR^n$ and the lattice spanned by the vertices with $\bZ^n$. Denote by $\Lambda^*\subset\kt^*$ the
lattice spanned by $p_{\nu}$ and the vertices of $P$. Use a translation to set $p_{\nu}=0$ (the vertices are still in $\Lambda^*$) and, up to a dilatation of $\nu$, assume that $L_l=\langle \nu_l,\cdot \rangle +1$, $l=1,\dots, d$. Since the
origin lies in the interior of the polytope, each facet contains a set of vertices of $P$ which is a (real) linear basis of $\kt^*$. Such a basis lies in $\Lambda^*$ by assumption and, thus, is a rational linear basis of $\Lambda^*\otimes \bQ$. This implies that $\langle \nu_l,q \rangle$ is a rational number for each $q\in\Lambda^*$ and $l=1,\dots, d$. In particular, there exists $m\in \bN$ such that the
set ${m\nu_1,\dots, m\nu_d}$ is included in the dual lattice of $\Lambda^*$. \end{proof}

\begin{corollary}\label{coroCOMMlatticePolytope} If $P$ is a polytope with vertices in $\Lambda^*$, and $\nu$ is a set of normals given by Theorem~\ref{theo1} then there exists a real number $s>0$ such that $\mbox{span}_{\bZ}\{s\nu\}$ is a sublattice of $\Lambda$, the dual lattice of $\Lambda^*$.
\end{corollary}

\subsection{Singular K\"ahler--Einstein metrics}\label{sesctSINGULAR}
Consider a Delzant labelled polytope $(P,\eta,\Lambda)$, see Definition~\ref{DelzantCOND}. The associated compact symplectic toric orbifold $(M,\omega,T)$ is a smooth manifold (all the orbifold structure groups are trivial). Denote the moment map $x: M\ra \kt^*$ and recall that $T=\quot{\kt}{\Lambda}$.\\

The construction of $(M,\omega,T)$ of Duistermaat--Pelayo~\cite{DuistPelayo}, recalled in Section~\ref{sectBACKGROUND}, allows us to see $M$ as a smooth compactification of $P\times T$. There is an equivariant symplectomorphism between the open subset of $M$ where the torus acts freely, say $\mathring{M}=x^{-1}(P)$, and $P\times T$.\\

Consider another set $\nu$ of normals inward to $P$ and a symplectic potential $u\in \mS(P,\nu)$. The K\"ahler metric $g_u$ defined by \eqref{ActionAnglemetric} is a $\kt$--invariant smooth K\"ahler metric on $P\times \kt$ and, thus, defines a smooth K\"ahler metric, still denoted $g_u$, on $P\times T$, compatible with the symplectic form $dx\wedge d\theta$. Hence, via the equivariant symplectic embedding $$(P\times T, dx\wedge d\theta, T)=(\mathring{M}, \omega_{|_{\mathring{M}}}, T)\subset (M,\omega,T),$$ $g_u$ is a smooth K\"ahler metric on $\mathring{M}$ compatible with the symplectic form $\omega_{|_{\mathring{M}}}$. This metric $g_u$ is not the restriction of a smooth metric on $M$ unless $\nu=\eta$.

\begin{rem} If $\nu$ spans a sub-lattice of $\Lambda$, the K\"ahler structure $(g_u,\omega_{|_{\mathring{M}}}, J_u)$ compactifies smoothly (in the orbifold sense) on the orbifold associated to $(P,\nu,\Lambda')$.\end{rem}
We now describe the singular behavior of $g_u$ along the toric submanifolds corresponding to the pre-image of the interior of the facets of $P$. For each facet $F_k$, there is a real number $a_k>0$ such that \begin{equation}\label{rapportSING} a_k\nu_k=\eta_k.\end{equation}

 The type of the  singularity along $x^{-1}(\mathring{F}_{k})$ only depends on $a_k$ and to understand which types of singularity may occur we only need to study the possibilities on a sphere. To see this, we can use the alternative definition of $\mS(P,\nu)$ of~\cite{H2FII}, recalled in \S\ref{subSectBC}. 

Let $P=(0,2)\subset \bR$ and $\nu_1= \frac{1}{a}\eta_1$ where $\eta_1$ is the generator of $S^1$. Thus, $(P,\{\eta_1, -\eta_1\})$ is the labelled polytope of the sphere $S^2$ of volume $4\pi$. Let a neighborhood of the origin $U\subset \bR^2$ be a $S^1$--equivariant chart around the south pole $x^{-1}(0)$. With respect to polar coordinates $(r,\theta)$ on $U$ (with the period $2\pi$ given by $\eta_1$ as in \eqref{localEQUIV}), we have $x= \frac{1}{2}r^2$ and $\theta=\theta$. On $(0,2)\times S^1$ the metric $g_u$ is
\begin{equation}\label{metricCYLINDRE}g_u=\frac{dx\otimes dx}{2ax +\bigO(x^2)} +(2ax +\bigO(x^2))d\theta\otimes d\theta.\end{equation}         
The change of coordinates in the metric~\eqref{metricCYLINDRE} gives
\begin{equation}\label{metricSINGULAR}
g_u = \frac{1}{a} \left(dr\otimes dr +a^2r^2d\theta\otimes d\theta \right) -\frac{r^2 dr\otimes dr}{a(a +\bigO(r^2))} + \bigO(r^4) d\theta\otimes d\theta
\end{equation} as computed in~\cite{H2FII}. The last two terms are smooth and vanish at $r=0$. Therefore
\begin{equation}\begin{split}\label{SINGbehavior}
 &\mbox{ if } a<1,\, g_u \,\mbox{ has a singularity of conical type and angle }2a\pi,\\
 &\mbox{ if } a=1,\, g_u \,\mbox{ is smooth},\\
 &\mbox{ if } a>1,\, g_u \,\mbox{ has a singularity caracterized by a large angle }2a\pi>2\pi.
\end{split}\end{equation}

Consequently, we obtain
\begin{proposition} \label{propSINGmetric} Let $(M,\omega, T)$ be a smooth compact symplectic toric manifold associated to the Delzant labelled polytope $(P,\eta,\Lambda)$. For any set $\nu$ of normals inward to $P$, the symplectic potentials in $\mS(P,\nu)$ define $T$--invariant, compatible, K\"ahler metrics on the open dense subset where the torus acts freely via formula \eqref{ActionAnglemetric}. The behavior of these metrics along the pre-image of the interior of the facet $\mathring{F}_k$ only depends on the real number $a_k>0$, defined by $a_k\nu_k=\eta_k$, as in \eqref{SINGbehavior}.\end{proposition}

%

\begin{proposition} \label{coroSING} Let $(M,\omega, T)$ be a smooth compact symplectic toric manifold associated to the Delzant labelled polytope $(P,\eta,\Lambda)$. Fix $\nu$, a set of inward normals such that $(P,\nu)$ is monotone and has a constant extremal affine function equals to $2n$.

For any $\lambda>0$, there exists a $T$--invariant K\"ahler--Einstein metric $g_{\lambda}$ smooth on the open dense subset where the torus acts freely, compatible with $\omega$ and with scalar curvature equals to $2n\lambda$. The type of singularity of $g_{\lambda}$ along the pre-image of the interior of the facet $\mathring{F}_k$ is one of the 3 cases of \eqref{SINGbehavior} with $a_k$ defined by $\lambda a_k\nu_k=\eta_k$. In particular, for $\lambda$ small enough, the singularity along the pre-image of the interior of any facet is of conical type.\end{proposition}

\begin{rem} Proposition~\ref{propSINGmetric} gives a way to construct plenty of singular metrics. For instance, a dilatation of the set of normals, $s\nu$ with $s>0$, corresponds in multiplying the volume (with respect to $g_u\in \mS(P,s\nu)$) of the orbits of $T$ in $M$ by a factor $s^{\frac{n}{2}}$. 
\end{rem}

Now, we interpret the K\"ahler metric $g_u$ has a singular K\"ahler metric in the usual sense (smooth complex structure and singular symplectic form). Indeed, the singularity of $\Phi_u^{-1}\circ\Phi_{u_o} : P\times \kt \ra P\times \kt$ (see notation \eqref{CPLXpointofview} in~\S\ref{subsectCPLX}) when $u\in\mS(P,\nu)$ and $u_o\in\mS(P,\eta)$ lies along the boundary $\del P$ and only depends on the ratio of $\eta$ and $\nu$. That is, $$((\Phi_u^{-1}\circ\Phi_{u_o})^*g_u, (\Phi_u^{-1}\circ\Phi_{u_o})^*dx\wedge d\theta, J_{u_o})$$ is a smooth K\"ahler metric on $P\times \kt =\mathring{M}$ compatible with $J_{u_o}$ and the complex structure $J_{u_o}$ admits a smooth compactification on $M$. We take back the local setting: $U\subset \bR^2$, a $S^1$--equivariant chart around the south pole $x^{-1}(0)$ as above, with the polar coordinates $(r,\theta)$ and singular K\"ahler structure $(g_u,J_u)$ given by~\eqref{metricSINGULAR} for some $a>0$. For our purpose (which is analysing the local singularity), the smooth K\"
ahler structure $(g_{u_o},dx\wedge d\theta,J_{u_o})$ can be identified to the standard one on $\bR^2\simeq\bC$, namely $(dr^2 +r^2d\theta^2,\omega=rdr\wedge d\theta, i)$, and the smooth part of $g_u$ can be forgotten. We put $g_u =\frac{1}{a}\left(dr\otimes dr +a^2r^2d\theta\otimes d\theta \right)$. Moreover, we don't need to find explicitly $\Phi_u^{-1}\circ\Phi_{u_o}$ but only a diffeomorphism of $U\backslash \{0\}$ which takes $g_u$ to a metric which is K\"ahler with respect to the standard complex structure on $U\backslash \{0\}$. Consider $z=r^{1/a}e^{i\theta}$ as in~\cite{Do:cone} to see that \begin{equation}\label{metricSINGULAR}a|z|^{2(a-1)}dz\otimes d\bar{z} = \frac{1}{a}\left(dr\otimes dr +a^2r^2d\theta\otimes d\theta \right) - idx\wedge d\theta =g_u -i\omega.\end{equation}

Up to the multiplicative factor $a$, the $(1,1)$--form $$\omega_a=a|z|^{2(a-1)} i dz\wedge d\bar{z}$$ is either: 
 \begin{equation}\begin{split}\label{SINGbehaviorCPLX}
 &\,\mbox{ singular of conical type and angle }2a\pi, \mbox{ when } a<1\\
 &\,\mbox{ smooth and positive definite, } \mbox{ when } a=1\\
 &\,\mbox{ smooth but not positive definite, } \mbox{ when } a>1.
\end{split}\end{equation}

\subsection{Singular K\"ahler--Einstein metrics on the first Hirzebruch surface}
Let $P$ be the convex hull of the points $(1,0), (1,1), (2,2), (2,0)$ in $\bR^2$. Consider the two sets of inward normals:

$$\eta=\left\{
\eta_1= \begin{pmatrix} 
1\\
 0
 \end{pmatrix},  \eta_2= \begin{pmatrix}
 -1\\
 0
 \end{pmatrix}, \eta_3= \begin{pmatrix}
 0\\
 1
 \end{pmatrix}, \eta_4= \begin{pmatrix} 
 1\\
 -1
 \end{pmatrix}
\right\},$$

\begin{equation*} \nu(C)=\left\{
\nu_1= C\frac{7}{5}\begin{pmatrix} 
1\\
 0
 \end{pmatrix},  \nu_2= C\frac{7}{4}\begin{pmatrix}
 -1\\
 0
 \end{pmatrix},\nu_3= C\begin{pmatrix}
  0\\
 1
  \end{pmatrix}, \nu_4= C\begin{pmatrix}
 1\\
 -1 \end{pmatrix}
\right\}.
\end{equation*}

One can check that $(P,\eta)$ satisfies the Delzant condition and corresponds to the first Hirzebruch surface $\bP(\mO+\mO(-1))$. On the other hand, $(P,\nu(C))$ is monotone and has constant extremal affine function. Actually, see~\cite{TGQ}, we explicitly know the form of the K\"ahler--Einstein metric on quadrilaterals in terms of the inverse of the Hessian of the potential: using notation~\eqref{ActionAnglemetric}, it reads
\begin{equation}\begin{split}                                                                                                                                                                                                                                                                                                                                                                     (H_{ij})= \frac{x_1}{x_1^2 - x_2}\begin{pmatrix} A(x_1)+ B(x_2/x_1) & (x_2/x_1)A(x_1)+ x_1B(x_2/x_1)\\(x_2/x_1)A(x_1)+ x_1B(x_2/x_1)&(x_2/x_1)^2A(x_1)+ x_1^2B(x_2/x_1)                                                                                                                                                                                                                                                                                                                                                                                                                                                                                       
                                                                                                                                                                                                                                                                                                                                            \end{pmatrix}                                                                                                                                                                                                                                                                                                                                                                        \end{split}                                                                                                                                                                                                                                                                                                                    
                      \end{equation}

with
\begin{equation}\begin{split}                                                                                                                                                                                                                                                                                                                                                                     &A(x) =  \frac{-2C}{7}(x-1)(x-2)(2+3x), \\
&B(y) = -2Cy(y-1).
\end{split}                                                                                                                                                                                                                                                                                                                \end{equation}

\begin{rem}\label{5/7} The case $C=1$ gives a singularity of angle $2\pi 5/7$ along the infinite section of $\bP(\mO+\mO(-1))$ which is precisely the angle of singularity obtained by Sz\'ekelyhidi~\cite{szekelyhidi} in the limit case of a construction (using Calabi ansatz) of metrics on $\bP(\mO+\mO(-1))$ satisfying $\mbox{Ric}(\omega)\geq \frac{6}{7}\omega$, see also~\cite{ChiLi0,ChiLi}.
\end{rem}

\bibliographystyle{abbrv}

\begin{thebibliography}{DDDD}


 \bibitem{abreu}
  {\scshape M. Abreu}
   \emph{K\"ahler geometry of toric varieties and extremal metrics}, Internat. J. Math. {\bf 9} (1998), 641--651.

 \bibitem{abreuOrbifold}
  {\scshape M. Abreu}
    \emph{K\"ahler metrics on toric orbifolds}, J. Differential Geom. {\bf 58} (2001), 151--187.


 \bibitem{H2FII}
  {\scshape V. Apostolov, D. M. J. Calderbank, P. Gauduchon, C. T{\o}nnesen-Friedman}
    \emph{Hamiltonian $2$--forms in K\"ahler geometry. II. Global classification}, J. Differential Geom. {\bf 68} (2004), 277--345.


 \bibitem{aubin}
   {\scshape T. Aubin}
     \emph{Equations de type Monge-Amp\`ere sur les vari\'et\'es k\"ahl\'eriennes compactes}, C. R. Acad. Sci. Paris {\bf 283} (1976), 119--121.


 \bibitem{BG:book}
  {\scshape C. P. Boyer,  K. Galicki}
    \emph{Sasakian geometry}, Oxford Mathematical Monographs. Oxford University Press, Oxford, 2008.

 \bibitem{BurnsGuil}
  {\scshape D. Burns, V. Guillemin}
\emph{Potential functions and action of tori on K\"ahler manifold}, Communication in Analysis and Geometry, {\bf 12}, 1, 281--303, 2004.

  \bibitem{calabiConj}
   {\scshape E. Calabi}
   \emph{The space of K\"aler metrics}, Proc. Internat. Congress Math. Amsterdam, Vol. {\bf 2} (1954), 206--207.


  \bibitem{calabi}
   {\scshape E. Calabi}
   \emph{Extremal K\"ahler metrics. II.}, Differential geometry and complex analysis, I. Chavel, H. M. Farkas Eds., 95--114, Springer, Berlin, 1985.

 \bibitem{CDG}
  {\scshape D.M.J. Calderbank, L. David, P. Gauduchon}
    \emph{ The Guillemin formula and K\"ahler metrics on toric symplectic manifolds}, J. Symp. Geom. 1 (2003) 767--784.


 \bibitem{CDS1}
  {\scshape X.-X. Chen, S.K. Donaldson, S. Sun}
    \emph{ K\"ahler--Einstein metrics on Fano manifolds, I: approximation of metrics with cone singularities}, arXiv/math.DG:1211.4566 



 \bibitem{CDS2}
  {\scshape X.-X. Chen, S.K. Donaldson, S. Sun}
    \emph{ K\"ahler--Einstein metrics on Fano manifolds, II: limits with cone angle less than $2\pi$}, arXiv/math.DG:1212.4714 



 \bibitem{CDS3}
  {\scshape X.-X. Chen, S.K. Donaldson, S. Sun}
    \emph{ K\"ahler--Einstein metrics on Fano manifolds, III: limits as cone angle approaches $2\pi$ and completion of the main proof}, arXiv/math.DG:1302.0282 


 \bibitem{delzant:corres}
  {\scshape T. Delzant}
    \emph{Hamiltoniens p\'eriodiques et images convexes de l'application moment}, Bull. Soc. Math. France {\bf 116} (1988), 315--339.

 \bibitem{don:scalar}
  {\scshape S. K. Donaldson}
      \emph{Constant scalar curvature metrics on toric surfaces}, Geom. Funct. Anal. {\bf 19} (2009), 83--136.

 \bibitem{don:Large}
  {\scshape S. K. Donaldson}
    \emph{ K\"ahler geometry on toric manifolds, and some other manifolds with large symmetry}, Handbook of geometric analysis. No. 1, Adv. Lect. Math. (ALM) {\bf 7}, 29--75, Int. Press, Somerville, MA, 2008.


  \bibitem{don:extMcond}
   {\scshape S. K. Donaldson}
     \emph{Extremal metrics on toric surfaces: a continuity method}, J. Differential Geom. {\bf 79} (2008), 389--432.

\bibitem{Do:cone}
   {\scshape S. K. Donaldson}
     \emph{K\"ahler metrics with cone singularities along a divisor}, Essays in mathematics and its applications, 49--79, Springer, Heidelberg, 2012.

\bibitem{DuistPelayo}
   {\scshape J.J. Duistermaat, A. Pelayo}
     \emph{Reduced phase space and toric variety coordinatizations of Delzant spaces.}
Math. Proc. Cambridge Philos. Soc. {\bf 146} (2009), no. 3, 695--718.

\bibitem{futaki}
  {\scshape A. Futaki}
      \emph{An obstruction to the existence of Einstein K\"ahler metrics}, Invent. Math. {\bf 73} (1983), no. 3, 437--443.

 \bibitem{FutakiOnoWang}
  {\scshape A. Futaki, H. Ono, G. Wang}
	\emph{Transverse K\"ahler geometry of Sasaki manifolds and toric Sasaki--Einstein manifolds},  J. Differential Geom. {\bf 83} (2009), 585--636.
	
 \bibitem{pg}
  {\scshape P. Gauduchon}
    \emph{Extremal K\"ahler metrics: An elementary inroduction}, In preparation.

 \bibitem{guillMET}
  {\scshape V. Guillemin}
    \emph{K\"ahler structures on toric varieties}, J. Diff. Geom. {\bf 40} (1994), 285--309.



 \bibitem{JoyceBOOK}
  {\scshape D.D. Joyce}
    \emph{Compact manifolds with special holonomy}, Oxford science publication 2000.

 \bibitem{lejmi}
  {\scshape M. Lejmi}
    \emph{Extremal almost-Kahler metrics}, Int. J. Maths {\bf 21} (2010), No. 12.


 \bibitem{TGQ}
  {\scshape E. Legendre}
    \emph{Toric geometry of convex quadrilaterals}, J. Symplectic Geom. {\bf 9} (2011), 343--385

  \bibitem{moi:ENU}
   {\scshape E. Legendre}
     \emph{Existence and non uniqueness of constant scalar curvature toric Sasaki metrics}, Compositio Math. {\bf 147} (2011), 1613--1634.


 \bibitem{LT:orbiToric}
  {\scshape E. Lerman, S. Tolman}
    \emph{Hamiltonian torus actions on symplectic orbifolds and toric varieties}, Trans. Amer. Math. Soc. {\bf 349} (1997), 4201--4230.


 \bibitem{L:contactToric}
  {\scshape E. Lerman}
    \emph{Toric contact manifolds}, J. Symplectic Geom. {\bf 1} (2003), 785--828.

 \bibitem{ChiLi0}
  {\scshape C. Li}
\emph{Greatest lower bounds on Ricci curvature for toric Fano manifolds}, Adv. Math. 226 (2011), {\bf 6}, 4921--4932. 

 \bibitem{ChiLi}
  {\scshape C. Li}
\emph{On the limit behavior of metrics in continuity method to K\"ahler--Einstein problem in toric Fano case},  Compos. Math. 148 (2012), {\bf 6}, 1985--2003.

 \bibitem{mabuchi}
  {\scshape T. Mabuchi}
    \emph{ Einstein-Kahler forms, Futaki invariants and convex geometry on toric Fano varieties}, Osaka J. Math. {\bf 24} (1987), 705-737.

  \bibitem{reebMETRIC}
   {\scshape D. Martelli,  J. Sparks, S.-T. Yau}
     \emph{The geometric dual of {\it a}-- maximisation for toric Sasaki-Einstein manifolds}, Comm. Math. Phys.  {\bf 268}  (2006), 39--65.

 \bibitem{schwarz}
   {\scshape G.W. Schwarz}, Smooth functions invariant under the action of a compact Lie group, Topology {\bf 14} (1975) 63--68,

   \bibitem{ast58}
  {\scshape S\'eminaire Palaiseau}
    \emph{Premiere classe de chern et courbure de Ricci positive: preuve de la conjecture de Calabi}, Ast\'erisque {\bf 58} (1978).
    
 \bibitem{SZ}
  {\scshape Y. Shi, X.H. Zhu}
    \emph{K\"ahler--Ricci solitons on toric Fano orbifolds}, preprint arXiv:math/1102.2764
    
 \bibitem{szekelyhidi}
   {\scshape G. Sz\'ekelyhidi}
 \emph{Greatest lower bounds on the Ricci curvature of Fano manifolds}, Compositio Math. {\bf 147} (2011), 319--331.


 \bibitem{TianProof}
  {\scshape G. Tian}
 \emph{$K$--stability and K\"ahler--Einstein metrics}, math.DG/1211.4669

 \bibitem{Tsingular}
  {\scshape G. Tian}
 \emph{K\"ahler--Einstein metrics on algebraic manifolds}, Transcendental methods in algebraic geometry (Cetraro 1994), Lecture Notes in Math. 1646, pp. 143--185.

 \bibitem{T}
  {\scshape G. Tian}
    \emph{ K\"ahler--Einstein metrics with positive scalar curvature}, Inventiones Math. {\bf 130} No 1 (1997), 1--37.

 \bibitem{TZ}
  {\scshape G. Tian, X.H. Zhu}
    \emph{Uniqueness of K\"ahler--Ricci solitons}, Acta Math. {\bf 184} (2000), 271-305.

 \bibitem{WZ}
  {\scshape X--J. Wang, X.H. Zhu}
    \emph{K\"ahler--Ricci solitons on toric manifolds with positive first Chern class}, Advances in Math. {\bf 188} (2004), 47--103.


 \bibitem{Y}
  {\scshape S.T. Yau}
    \emph{On the Ricci curvature of a compact K\"ahler manifold and the complex Monge-Amp\`ere equation I}, Comm. Pure Appl. Math., {\bf 31} (1978), 339-411.

 \bibitem{Z}
  {\scshape X.H. Zhu}
    \emph{K\"ahler--Ricci soliton typed equation on compact complex manifolds with $c_1(M)>0$}, J. Geom. Anal. {\bf 10} (2000), 759--774.



\end{thebibliography}

\end{document}